\documentclass[reqno,11pt]{amsart}
\usepackage[colorlinks=true, linkcolor=blue, citecolor=blue]{hyperref}

\usepackage{amssymb}
\usepackage{amsmath, graphicx, rotating}
\usepackage{color}
\usepackage{soul}
\usepackage[dvipsnames]{xcolor}

\usepackage{ifthen}
\usepackage{xkeyval}
\usepackage{todonotes}
\usepackage[T1]{fontenc}
\usepackage{lmodern}
\usepackage[english]{babel}

\usepackage{ upgreek }
\usepackage{stmaryrd}
\SetSymbolFont{stmry}{bold}{U}{stmry}{m}{n}
\usepackage{amsthm}
\usepackage{float}

\usepackage{ bbm }
\usepackage{ stmaryrd }
\usepackage{ mathrsfs }
\usepackage{ frcursive }
\usepackage{ comment }

\usepackage{pgf, tikz}
\usetikzlibrary{shapes}
\usepackage{varioref}
\usepackage{enumitem}

\definecolor{rouge}{rgb}{0.7,0.00,0.00}
\definecolor{vert}{rgb}{0.00,0.5,0.00}
\definecolor{bleu}{rgb}{0.00,0.00,0.8}
\usepackage[margin=1in]{geometry}
\newtheorem{theorem}{Theorem}[section]
\newtheorem*{theorem*}{Theorem}
\newtheorem{lemma}[theorem]{Lemma}
\newtheorem{definition}[theorem]{Definition}

\newtheorem{proposition}[theorem]{Proposition}

\labelformat{hypothesis}{\textbf{M\kern-0.1mm#1}}

\newtheorem{condition}{Condition}

\newtheorem{conditionA}{A\kern-0.1mm}
\labelformat{conditionA}{\textbf{A\kern-0.1mm#1}}

\labelformat{conditionB}{\textbf{B\kern-0.1mm#1}}

\theoremstyle{definition}

\newtheorem{remark}[theorem]{Remark}

\def \eref#1{\hbox{(\ref{#1})}}

\numberwithin{equation}{section}

\def\geq{\geqslant}
\def\leq{\leqslant}

\def\RR{\mathbb{R}}
\def\PP{\mathbb{P}}
\def\EE{\mathbb{E}}

\def\vare{{\varepsilon}}
\def \eref#1{\hbox{(\ref{#1})}}

\def\EE{\mathbb{ E}}

\allowdisplaybreaks
\begin{document}

\title[Strong averaging principle for nonautonomous slow-fast SPDEs]
{Strong averaging principle for nonautonomous slow-fast SPDEs driven by $\alpha$-stable processes}

\author{Yueling Li\quad }
\curraddr[Li, Y.]{School of Mathematics and Statistics/RIMS, Jiangsu Normal University, Xuzhou, 221116, P.R. China}
\email{lylmath@jsnu.edu.cn}

\author{Xiaobin Sun\quad }
\curraddr[Sun, X.]{School of Mathematics and Statistics/RIMS, Jiangsu Normal University, Xuzhou, 221116, P.R. China}
\email{xbsun@jsnu.edu.cn}
	
\author{\quad Zijuan Wang\quad }
\curraddr[Wang, Z.]{School of Mathematics and Statistics, Jiangsu Normal University, Xuzhou, 221116, P.R. China}
\email{zjwang@jsnu.edu.cn}
	
\author{\quad Yingchao Xie}
\curraddr[Xie, Y.]{School of Mathematics and Statistics/RIMS, Jiangsu Normal University, Xuzhou, 221116, P.R. China}
\email{ycxie@jsnu.edu.cn}

\begin{abstract}
This paper considers a class of nonautonomous slow-fast stochastic partial differential equations driven by $\alpha$-stable processes for $\alpha\in (1,2)$. By introducing the evolution system of measures, we establish an averaging principle for this
stochastic system. Specifically, we first prove the strong convergence (in the $L^p$ sense for $p\in (1,\alpha)$) of the slow component to the solution of a simplified averaged equation with coefficients depend on the scaling parameter. Furthermore, under conditions that coefficients are time-periodic or satisfy certain asymptotic convergence, we prove that the slow component converges strongly to the solution of an averaged equation, whose coefficients are independent of the scaling parameter. Finally, a concrete example is provided to illustrate the applicability of our assumptions. Notably, the absence of finite second moments in the solution caused by the $\alpha$-stable processes requires new technical treatments, thereby solving a problem mentioned in \cite[Remark 3.3]{BYY2017}.
\end{abstract}

\date{\today}

\maketitle

\section{Introduction}

\subsection{Background}
Let $ H $ be a given Hilbert space, with $ \langle\cdot, \cdot\rangle $ and $ |\cdot| $ denoting its inner product and norm, respectively. This paper focus on the following slow-fast nonautonomous stochastic partial differential equations (SPDEs for short) driven by cylindrical stable processes:
\begin{equation}\label{main equation 1}
 \left\{\begin{array}{l}d X_t^{\varepsilon}=\left[A X_t^{\varepsilon}+F\left(t / \varepsilon, X_t^{\varepsilon}, Y_t^{\varepsilon}\right) d t+d L_t,\quad X_0^{\varepsilon}=x \in H,\right. \\\\
 d Y_t^\varepsilon=\frac{1}{\varepsilon}\left[BY_t^{\varepsilon}+G\left(t / \varepsilon, X_t^{\varepsilon}, Y_t^{\varepsilon}\right)\right] d t+\frac{1}{\varepsilon^{1 / \alpha}} d Z_t, \quad Y_0^{\varepsilon}=y \in H,\end{array}\right.
\end{equation}
where $\{L_t\}_{t\geq 0}$ and $\{Z_t\}_{t\geq 0}$ are two mutually independent cylindrical $\alpha$-stable processes with $\alpha\in(1,2)$ defined on probability space $(\Omega, \mathscr{F},\mathbb{P})$ with the natural filtration $\{\mathscr{F}_t, t\ge 0\}$ generated by $(L_t)_{t\ge  0}$ and $(Z_t)_{t\ge0}$, that is
$$
L_t=\sum^{\infty}_{k=1}\rho_{k}L^{k}_{t}e_k,\quad Z_t=\sum^{\infty}_{k=1}\gamma_{k}Z^{k}_{t}e_k,\quad t\geq 0,
$$
where $\{\rho_k\}_{k\ge 1}$ and $\{\gamma_k\}_{k\ge 1}$ are two given sequences of positive numbers, $\{e_k\}_{k\ge 1}$ is an orthonormal basis of $H$, $\{L^k_t\}_{k\ge 1}$ and $\{Z^k_t\}_{k\ge 1}$
are two sequences of independent one dimensional symmetric $\alpha$-stable processes satisfying that for any $k\ge 1$ and $t\geq0$,
$$\mathbb{E}[e^{i L^k_{t}h}]=\mathbb{E}[e^{i Z^k_{t}h}]=e^{-t|h|^{\alpha}}, \quad \forall h\in \mathbb{R}.$$

The aforementioned stochastic system \eqref{main equation 1} involves a parameter $\varepsilon$. Intuitively, as $\varepsilon\to 0$, the components $X^\varepsilon$ and $Y^\varepsilon$ of the solution exhibit slow and fast changes, respectively. Consequently, $X^\varepsilon$ and $Y^\varepsilon$ can be regarded as the "slow" and "fast" components in the system. This kind of slow-fast system has appeared in many fields, such as the nonlinear oscillations, chemical kinetics, biology, climate dynamics, see e.g. \cite{BR2017,EE2003,PS2008,WTRY2016}. Take a consideration of the complexity of such coupled system, it is crucial to find simplified equation which can serve as approximate replacements for the original system. Thus scholars are interested in studying the asymptotic behavior of the slow component as $\vare\to 0$, which is called averaging principle. The averaging principle for slow-fast SPDEs was first studied by Cerrai and Freidlin \cite{C2009,CF2009} in 2009. Over the past fifteen years, a large number of related results has been achieved in this field, see e.g. \cite{B2012,CL2023,DSXZ2018,FLL2015,GS2024,LRSX2023,PXY2017,SG2022,WR2012,XML2015}.

It is well-known that the $\alpha$-stable process has theoretically meaningful, for instance such processes can be used to model systems with heavy tails in physics, telecommunication networks, finance and other fields. Numerous scholars are dedicated to this research field of SPDEs driven by $\alpha$-stable process, see e.g. \cite{DXZ2014,PZ2011,Xu2013}. Recently, the slow-fast SPDEs driven by $\alpha$-stable process has attracted people's attention, for instance,  Bao et al. \cite{BYY2017} proved the strong averaging principle for two-time scale SPDEs driven by $\alpha$-stable noise. The authors have proved the strong averaging principle for slow-fast stochastic Ginzburg-Landau equation and stochastic Burgers equations in \cite{SZ2020} and \cite{CSS2020}, respectively. Moreover, using the technique of Poisson equation, we also prove the strong and weak convergence rates for multi-scale stochastic differential equations (SDEs) and SPDEs driven by $\alpha$-stable processes in \cite{SXX2022} and \cite{SX2023}, respectively.

Note that aforementioned references all consider the autonomous (time-independent) SPDEs, this focus is due to the extensive theoretical work on the existence and uniqueness of invariant measure for the corresponding time-independent frozen equation. It is important to mention that the stochastic system \eref{main equation 1} is nonautonomous, that is, its coefficients will change over time due to external factors or internal variations, such as the learning models related to neuronal activity in \cite{GW2012}. Consequently, the concept of invariant measure does not exist for the corresponding time-dependent frozen equation.

To our knowledge, there are few studies on the slow-fast non-autonomous S(P)DEs. Regarding the slow-fast nonautonomous  SDEs, Wainrib \cite{W2013} and Uda \cite{U2021} proved the strong averaging principle when the coefficients of the fast equation are time-periodic. More recently, Sun et al. utilized the approach of nonautonomous Poisson equations, as presented in \cite{SWX2024}, and Khasminskii's time-discretization technique, outlined in \cite{SWX2025}, to investigate the averaging principles in both strong and weak senses. While in the framework of the slow-fast nonautonomous SPDE, there appears to be only one finding by Cerrai and Lunardi in \cite{CL2017}, where they explored its averaging principle when the coefficients in the fast equation satisfy the almost periodic condition. Nevertheless, the existing literature focus on the Wiener noise. It appears that there is a lack of research concerning the cases where the driven noises are $\alpha$-stable processes. Therefore, the primary aim of this paper is to investigate the averaging principle for stochastic system \eqref{main equation 1}.

\subsection{Outline of the methods}

The initial challenge involves comprehending the limit behaviour for the corresponding frozen equation
\begin{align*}
dY_{t}=\left[BY_{t}+G(t,x,Y_{t})\right]dt+d Z_{t},
\end{align*}
whose transition semigroup $\{P^x_{s,t}\}_{t\geq s}$ is time-inhomogeneous. Thus we need to introduce an evolution system of measures for $\{P^x_{s,t}\}_{t\geq s}$ of time-inhomogeneous frozen equation, see e.g. \cite{CL2021,DR2006,DR2008}. Specially, we call $\{\mu^x_t\}_{t\in \RR}$ is an evolution system of measures for $\{P^x_{s,t}\}_{t\geq s}$, if
$$
\int_{H}P^x_{s,t} \varphi(y)\,\mu^x_s(dy)=\int_{H}\varphi(y)\,\mu^x_t(dy),\quad
s\leq t,\varphi\in C_b(H).
$$
Therefore, it is reasonable to define the averaged coefficient by averaging the drift coefficient $F$ with respect to the evolution system of measures $\{\mu^x_t\}_{t\in \RR}$, i.e.,
\begin{equation}\label{e:drift}
\bar{F}(t,x)=\int_{H}F(t,x,y)\,\mu^x_t(dy).
\end{equation}
Thus the corresponding averaged equation can be expressed in the following form:
\begin{equation}\label{AV1}
d\bar{X}^{\varepsilon}_{t}=\left[AX^{\varepsilon}_t+\bar{F}(t/\varepsilon,\bar{X}^{\varepsilon}_t)\right]dt+dL_t,\quad\bar{X}^{\varepsilon}_{0}=x.
\end{equation}
Hence under some proper conditions, the first result of this paper is as follows:
 \begin{equation*}
\lim_{\varepsilon \rightarrow 0}\sup_{t\in [0, T]}\EE|X_{t}^{\varepsilon}-\bar{X}^{\varepsilon}_{t}|^{p}=0,\quad \forall p\in (1,\alpha), ~T>0.
\end{equation*}

It should be noted that even if we assume that $F$ is time-independent,
the time-dependent $\{\mu_t^x\}_{t\in\RR}$ still implies that $\bar{F}$ definition in \eqref{e:drift} has a time oscillating component.

Note that the coefficients in the above averaged equation \eqref{AV1} still include the parameter $\vare$. Therefore, it seems natural to further simplify \eqref{AV1} by averaging over time, while the critical step is sufficient to find a proper function $\bar{F}:H\to H$ such that
\begin{align}
\lim_{T\rightarrow\infty}\sup_{t\geq 0}\frac{1}{T}\int_{t}^{t+T}[\bar{F}(s,x)-\bar{F}(x)]ds=0.\label{e:Cb}
\end{align}
This is a ordinary condition in the studying the averaging principle for stochastic system \eqref{AV1}, as evidenced in references such as \cite[(G1)]{CL2023}. To achieve \eqref{e:Cb}, we here additionally assume that the coefficients $F$ and $G$  satisfy time periodicity or certain asymptotic convergence, respectively.

\subsubsection{Time periodic case} In accordance with the periodic conditions established in \cite{GW2012,U2021}. We suppose that the coefficients $ F(\cdot, x, y) $ and $ G(\cdot, x, y) $ are $\tau_1$-periodic and $\tau_2$-periodic, respectively, where $\tau_1/\tau_2$ is a rational number. Roughly speaking, the $\tau_2$-periodicity of $G(\cdot, x, y)$ induces that $\{\mu^x_t\}_{t\in\RR}$ is also $\tau_2$-periodic, which combines with the $\tau_1$-periodicity of $ F(\cdot, x, y) $  imply the $\tau$-periodicity of $\bar F(t,x)$ defined in \eqref{e:drift} with $\tau$ satisfies $\tau=m_2\tau_1=m_1\tau_2$, for some $m_1,m_2\in\mathbb{N}_{+}=\{1,2,\cdots\}$. In this case, the natural definition for the averaged coefficient is given by
\begin{align}
\bar{F}_P(x):=\frac{1}{\tau}\int^{\tau}_0 \bar F(t,x)dt.\label{FP}
\end{align}
It is easy to check that \eqref{e:Cb} holds for $\bar{F}(x)=\bar{F}_P(x)$. Consequently, the associated averaged equation can be expressed as
\begin{align}
d\bar{X}_t=A\bar{X}_t dt+\bar{F}_P(\bar{X}_t)dt + dL_t,\quad \bar{X}_0 = x\in H.\label{AV2}
\end{align}
Therefore, the second result in this paper is as follows:
 \begin{equation*}
\lim_{\varepsilon \rightarrow 0}\sup_{t\in [0, T]}\EE|X_{t}^{\varepsilon}-\bar{X}_{t}|^{p}=0,\quad \forall p\in (1,\alpha), ~T>0.
\end{equation*}

\subsubsection{Asymptotic convergence case}
In contrast to the periodic conditions previously imposed on the coefficients $F$ and $G$, we now introduce asymptotic convergence conditions. Specifically, there exist functions $\tilde{F},\tilde{G}:H\times H\rightarrow H$ such that for any $T\geq 0$, $x,y\in H$, the following limits hold
\begin{align*}
\lim_{T\to \infty}\sup_{t\geq 0}\frac{1}{T}\left|\int^{t+T}_t \left[F(s,x,y)-\tilde{F}(x,y)\right]ds\right|=0,\quad \lim_{T\to \infty}|G(T,x,y)-\tilde{G}(x,y)|=0.
\end{align*}
These conditions ensure that \eqref{e:Cb} is satisfied for
\begin{align}
\bar{F}(x)=\bar{F}_A(x):=\int_{H}\tilde F(x,y)\,\mu^x(dy),\label{FA}
\end{align}
where $\mu^x$ is the unique invariant measure associated with the following SPDE:
\begin{equation*}
d\tilde Y_{t}=B\tilde Y_{t} dt+\tilde{G}(x,\tilde Y_{t})\,dt+d Z_t.
\end{equation*}
As a result, the third result of this paper is as follows:
 \begin{equation*}
\lim_{\varepsilon \rightarrow 0}\sup_{t\in [0, T]}\EE|X_{t}^{\varepsilon}-X_{t}|^{p}=0,\quad \forall p\in (1,\alpha), ~T>0,
\end{equation*}
where $X$ is the solution of the averaged equation:
\begin{align}
dX_t=AX_t dt+\bar{F}_A(X_t)dt + dL_t,\quad X_0 = x\in H.\label{AV3}
\end{align}

\subsection{Contributions}

The initial contribution of this paper extends the framework of autonomous multi-scale SPDEs driven by $\alpha$-stable processes (see e.g. \cite{BYY2017,SX2023}) to the more general nonautonomous framework. It also extends nonautonomous SPDEs driven by Brownian motion (see e.g. \cite{CL2017}) to the nonautonomous SPDEs driven by $\alpha$-stable process.
Another contribution of this paper is to fill a gap as stated in \cite[Remark 3.3]{BYY2017}, that is \emph{"for the technical reason, it seems hard to show Theorem 3.1 without the uniform boundedness of the nonlinearity"}, where \emph{"Theorem 3.1"} means the strong averaging principle holds and \emph{"the nonlinearity"} means the coefficient $F$.
Although it has been mentioned about of filling this gap partially in \cite{SX2023}, using the technique of Poisson equation but some bounded conditions of second and third derivatives for the coefficients are assumed. Nevertheless, some technical treatments need to be introduced to effectively eliminate the uniform boundedness of $F$.

\vspace{2mm}
The remainder of the paper is structured as follows. In section 2, we provide detailed assumptions on the coefficients of the system under consideration. We also state our main results. In section 3, we establish a priori estimates of the solution and an auxiliary process. In section 4, we study the evolution system of measurable for the time-dependent SPDEs driven by $\alpha$-stable process. Section 5 is devoted to giving the detailed proofs of the main results. Finally, an special example is given to illustrate the effectiveness of all conditions.

Throughout this paper, we use $C$ and $C_{T}$ to represent constants whose values may vary from line to line. We use $C_{T}$ (resp. $C_{p,T}$) to emphasize that the constant depends on $T$ (resp. $p$ and $T$).

\section{Assumptions and main results}

We suppose the following assumptions:
\begin{conditionA}\label{A1}
Assume $A: \mathscr{D}(A) \subset H \rightarrow H$ and $B: \mathscr{D}(B) \subset H \rightarrow H$ are two self-adjoint linear operators on $ H $ such that the fixed basis $\{e_k\}_{k=1}^\infty \subset \mathscr{D}(A)\cap \mathscr{D}(B)$ of $H$ satisfies
$$ A e_k = -\lambda_k e_k,\quad B e_k = -\beta_k e_k ,$$
where $ \lambda_k,\beta_k > 0 $ and $ \lambda_k\uparrow \infty $, $ \beta_k\uparrow \infty $ as $ k \to \infty $.
\end{conditionA}

\begin{conditionA}\label{A2} There exists $\theta\in (0,2)$ such that
\begin{align}
\sum^{\infty}_{k=1}\frac{\rho^{\alpha}_k}{\lambda^{1-\alpha\theta/2}_k}<\infty,\quad \sum^{\infty}_{k=1}\frac{\gamma^{\alpha}_k}{\beta_k}<\infty.
\end{align}
\end{conditionA}

\begin{conditionA}\label{A3}
Assume that $H$-valued functions $ F$ on $ \RR_{+}\times H\times H $ and $ G$ on $\RR \times H \times H$ are measurable and there exist constants $C>0, L_{G}<\beta_1$ such that for any $x_i,y_i\in H$,  $i=1,2$,
\begin{align}\label{LipFG}
\begin{split}
&\sup_{t\geq 0}\left|F(t,x_1,y_1)-F(t,x_2,y_2)\right|\leq C\left(|x_1-x_2|+|y_1-y_2|\right),\\
&\sup_{t\in\RR}\left|G(t,x_1,y_2)-G(t,x_2,y_2)\right|\leq C|x_1-x_2|+L_{G}|y_1-y_2|,\\
&\sup_{t\geq 0}|F(t,x_1,y_1)|\leq C(1+|x_1|+|y_1|),\quad \sup_{t\in\RR}|G(t,x_1,y_1)|\leq C\left(1+|x_1|\right)+L_G|y_1|.
\end{split}
\end{align}
\end{conditionA}

\begin{remark}\label{Re1}
By a minor revision in the proof of \cite[Theorem 5.3]{PZ2011}, under Assumptions \ref{A1}-\ref{A3}, it is easy to prove that for any given $\varepsilon>0$ and  $x, y\in H$, the equation \eref{main equation 1} admits a unique mild solution $(X^{\varepsilon}_t, Y^{\varepsilon}_t)$, that is, $\PP$-a.s.,
\begin{equation}\left\{\begin{array}{l}\label{A mild solution}
\displaystyle
X^{\varepsilon}_t=e^{tA}x+\int^t_0e^{(t-s)A}F(s/{\varepsilon},X^{\varepsilon}_s, Y^{\varepsilon}_s)ds+\int^t_0 e^{(t-s)A}dL_s,\vspace{2mm}\\
\displaystyle
Y_t^\varepsilon=e^{tB/\varepsilon} y+\frac{1}{\varepsilon} \int_0^t e^{(t-s)B/\varepsilon} G\left(s / \varepsilon, X_s^\varepsilon, Y_s^\varepsilon\right) d s+\frac{1}{\varepsilon^{1/\alpha}} \int_0^t e^{(t-s)B/\varepsilon} d Z_s.
\end{array}\right.
\end{equation}
\end{remark}

\begin{remark}\label{Re2}
Note that the time domain of the function $G$ is the entire $\RR$ because we need to consider the evolution system of measures for the frozen equation in Section \ref{Section4}. In contrast, the time domain of the function $F$ is limited to $\RR_{+}$, which is sufficient for the analysis of the averaging principle over the interval $[0,T]$, for any $T>0$.
\end{remark}

\vspace{0.2cm}

For any $ s \in \mathbb{R} $, define
$$
H^s := \mathscr{D}\left((-A)^{s / 2}\right) := \left\{ u = \sum_{k=1}^{\infty} u_k e_k : u_k \in \mathbb{R}, \sum_{k=1}^{\infty} \lambda_k^s u_k^2 < \infty \right\},
$$
$$
(-A)^{s / 2} u := \sum_{k=1}^{\infty} \lambda_k^{s / 2} u_k e_k, \quad u \in \mathscr{D}\left((-A)^{s / 2}\right),
$$
and its corresponding norm
$$
\|u\|_s^2 := \left|(-A)^{s / 2} u\right|^2=\sum_{k=1}^{\infty} \lambda_k^s u_k^2 .
$$
It is easy to see that $ \|\cdot\|_0 = |\cdot| $. The operator semigroup $ e^{t A} $ has the following properties (see e.g. \cite[Proposition 2.4]{B2012})
\begin{align}\label{rr}
\begin{split}
& \left|e^{t A} x\right| \leq e^{-\lambda_1 t}|x|,\quad x\in H, t \geq 0 \\
& \left\|e^{t A} x\right\|_{\sigma_2} \leq C_{\sigma_1, \sigma_2} t^{-\frac{\sigma_2-\sigma_1}{2}} e^{-\frac{\lambda_1 t}{2}}\|x\|_{\sigma_1}, \quad x \in H^{\sigma_2}, \sigma_1<\sigma_2, t>0, \\
&\left|e^{t A} x-x\right| \leq C_\sigma t^{\frac{\sigma}{2}}\|x\|_\sigma, \quad x \in H^\sigma, \sigma>0, t \geq 0, \\
& \left|e^{t A} x-e^{s A} x\right| \leq C_\sigma(t-s)^\sigma s^{-\sigma}|x|, \sigma \in(0,1), 0<s \leq t, x \in H.
\end{split}
\end{align}
Similarly,  for the operator semigroup $ e^{t B} $, one has
$$|e^{t B}y|\leq e^{-\beta_1 t}|y|,\quad t\geq 0, y\in H.$$

By \cite[Lemma 4.1]{PSXZ2012}, if $\sum_{k=1}^{\infty} \rho_k^\alpha/\lambda_k^{1-\alpha \theta / 2}<\infty$ for some $\theta \geqslant 0$,  then for any $s\geq 0$ and $0<p<\alpha$, we have
\begin{align}
 \sup _{t \geqslant s} \mathbb{E}\left\|\int_s^t e^{(t-r) A} d L_r\right\|_\theta^p \leqslant C_{\alpha, p}\left(\sum_{k=1}^{\infty} \frac{\rho_k^\alpha}{\lambda_k^{1-\alpha \theta / 2}}\right)^{p/ \alpha}.\label{LAtheta}
\end{align}
And if $\sum_{k=1}^{\infty} \gamma_k^\alpha/\beta_k<\infty$,
we have that for any $s\geq 0$ and $0<p<\alpha$,
\begin{align}
 \sup _{t \geqslant s} \mathbb{E}\left|\int_s^t e^{(t-r) B} d Z_r\right|^p \leqslant C_{\alpha,p}\left(\sum_{k=1}^{\infty} \frac{\gamma_k^\alpha}{\beta_k}\right)^{p/ \alpha}.\label{ZA}
\end{align}

\vspace{0.1cm}
We present the first main result of this paper.

\begin{theorem}\label{MR1}
Suppose that Assumptions \ref{A1}-\ref{A3} hold. Then for any $p\in (1,\alpha)$ and $ T > 0 $, there exists $C_{p,T}>0$ such that all $x, y \in H $ and $\vare\in (0,1]$ we have
\begin{align}
\sup _{t \in[0, T]}\mathbb{E}\left|X_t^{\varepsilon}-\bar{X}_t^{\varepsilon}\right|^p\leq C_{p,T}(1+|x|^p+|y|^p)\varepsilon^{\frac{\theta(p-1)}{\theta(p-1)+2}}, \label{R1}
\end{align}
where $\bar{X}_t^{\varepsilon}$ is the unique solution of the averaged equation \eqref{AV1}.
\end{theorem}

\begin{remark}
(i) According to \eqref{R1}, if we choose $p$ to be very close to $\alpha$ and $\theta$ to be close to $2$, we can achieve a convergence order that is close to $\frac{1}{\alpha}\left(1-\frac{1}{\alpha}\right)$. Although this is not the optimal order, it might be possible to achieve the optimal convergence rate by leveraging technique of nonautomuous Poisson equation, which is left in our future work.

(ii) Note that when $F$ and $G$ are time-independent, the conditions in Theorem \ref{MR1} only require $F$ to be Lipschitz continuous (and thus linear growth), but do not impose boundedness on $F$. Therefore, the obtained conclusion provides a good answer of an unsolved problem mentioned in \cite[Remark 3.3]{BYY2017}, that is \emph{"for the technical reason, it seems hard to show Theorem 3.1 without the uniform boundedness of the nonlinearity"}, where \emph{"Theorem 3.1"} means the strong averaging principle holds and \emph{"the nonlinearity"} means the coefficient $F$. Here we will use some technical treatments to remove the uniform boundedness of $F$ effectively.
\end{remark}

To ensure that the averaged equation \eqref{AV1} is independent of the parameter $\varepsilon$, additional conditions on the coefficients are required. Here we mainly focus on the two cases: (1) the coefficients are time periodic; (2) the coefficients possess certain asymptotic properties.

\vspace{0.2cm}
For time periodic coefficient case,  one needs the followings.

\begin{conditionA}\label{A4}
Assume that $ F(\cdot, x, y) $ and $ G(\cdot, x, y) $ are $\tau_1$-periodic and $\tau_2$-periodic, respectively, i.e.,
\begin{align}\label{p}
\begin{split}
F(t + \tau_1, x, y) = F(t, x, y),\quad 
G(t +\tau_2, x, y) = G(t, x, y),\quad\forall t\in\RR, x, y \in H,
\end{split}
\end{align}
where $\tau_1,\tau_2>0$ and $\tau_1/\tau_2$ is a rational number.
\end{conditionA}

The following is our second main result.

\begin{theorem}\label{MR2}
Suppose that Assumptions \ref{A1}-\ref{A4} hold. Then for any $p\in (1,\alpha)$ and $T>0$, there exists $C_{p,T}>0$ such that all $x, y \in H $ and $\vare\in (0,1]$ we have
\begin{align}\label{K}
\sup _{t \in[0, T]}\mathbb{E}\left|X_t^{\varepsilon}-\bar{X}_t\right|^p\leq C_{p,T}(1+|x|^p+|y|^p)\varepsilon^{\frac{\theta(p-1)}{\theta(p-1)+2}},
\end{align}
where $\bar{X}$ is the unique solution of the averaged equation \eqref{AV2}.
\end{theorem}

\vspace{0.2cm}

For asymptotic properties on the coefficients, we need the following.

\begin{conditionA}\label{A5}
There exist $H$-valued functions $\tilde{F}$ and $\tilde{G}$ on $H\times H$, and locally bounded functions $\phi_i\ge 0$ on $ [0,\infty)$ ($i=1,2$) such that for any $T\geq 0$, $x,y\in H$,
\begin{align}\label{Cb}
\sup_{t\geq 0}\frac{1}{T}\left|\int^{t+T}_t \left[F(s,x,y)-\tilde{F}(x,y)\right]ds\right|\leq \phi_1(T)(1+|x|+|y|)
\end{align}
and
\begin{align}
|G(T,x,y)-\tilde{G}(x,y)|\leq \phi_2(T)(1+|x|+|y|),\label{Cf}
\end{align}
where $\phi_i(T)\to 0$ as $T\to \infty$, $i=1,2$.
\end{conditionA}

Thus we state our third main result.

\begin{theorem}\label{MR3}
Suppose the Assumptions \ref{A1}-\ref{A3} and \ref{A5} hold. Then for any $p\in (1,\alpha)$ and $T>0$, there exists $C_{p,T}>0$ such that all $x, y \in H $ and $\vare\in (0,1]$ we have
\begin{align}
\sup _{t \in[0, T]}\!\mathbb{E}\left|X_t^\varepsilon \!-\! X_t\right|^p\!\leq\!  C_{p,T}(1\! +\! |x|^p\!+\! |y|^p)\!\left[\varepsilon^{\frac{\theta(p-1)}{\theta(p-1)+2}}
+\left(\phi_1(\vare^{-\theta/(\theta+2)})+\tilde\phi_2(\vare^{-\theta/(\theta+2)})\right)^p\right],\label{R2}
\end{align}
where $X$ is the unique solution of the averaged equation \eqref{AV3}, and
\begin{align}
\tilde\phi_2(T):=\sup_{t\geq 0}\frac{1}{T}\int^{t+T}_t \int^s_0 e^{-\beta_1(s-u)}\phi_2(u)\,duds.\label{phi2}
\end{align}
\end{theorem}

\begin{remark}
Estimate \eqref{R2} indicates that the convergence rate of $X_t^\varepsilon$ to $X_t$ in the strong sense (as $\varepsilon$ approach $0$) is influenced by how quickly $\phi_1(T)$ and $\phi_2(T)$ approach 0 (as $T\to\infty$).
\end{remark}

\vspace{0.2cm}

\section{Some a priori estimates and auxiliary process}

In this section, we first establish a priori estimates for the solution process
$\{(X^{\varepsilon}_t,Y^{\varepsilon}_t)\}_{t\geq0}$.
Secondly, we construct an auxiliary process
$\{\hat{Y}_{t}^{\varepsilon}\}_{t\geq0}$ and estimate the difference process
$\{Y^{\varepsilon}_t-\hat{Y}_{t}^{\varepsilon}\}_{t\geq0}$.

\subsection {Some a priori estimates}

\begin{lemma}\label{L3.1}
For any $T>0$ and $p\in [1,\alpha)$, there exists a constant $C_{T}>0$ such that
\begin{align}
\sup_{\vare>0,t\in [0,T] }\left(\mathbb{E}|X_{t}^{\vare}|^{p}\right)^{1/p}\leq C_{T}(1+|x|+|y|)\label{EX}
\end{align}
and
\begin{align}
\sup_{\vare>0,t\in [0,T]}\left(\mathbb{E}|Y_{t}^{\varepsilon}|^{p}\right)^{1/p}\leq C_{T}(1+|x|+|y|).\label{EY}
\end{align}
\end{lemma}
\begin{proof}
Define $\tilde Z_t:=\frac{1}{\vare^{1/ \alpha}}Z_{t\vare}$, which is also a cylindrical $\alpha$-stable process. Then by \eqref{ZA}, for any $p\in [1,\alpha)$,
\begin{align*}
\EE\left|\frac{1}{\vare^{1/\alpha}}\int^t_0 e^{(t-s)B/\varepsilon}dZ_s\right|^p=\EE\left|\int^{t/\vare}_0 e^{(t/\vare-s)B}d\tilde Z_s\right|^{p}
\leq C\left(\sum^{\infty}_{k=1}\frac{\gamma^{\alpha}_k }{\beta_k} \right)^{p/\alpha}<\infty.
\end{align*}

Recall that
\begin{align*}
Y_t^\varepsilon=e^{tB/\varepsilon} y+\frac{1}{\varepsilon} \int_0^t e^{(t-s)B/\varepsilon} G\left(s / \varepsilon, X_s^\varepsilon, Y_s^\varepsilon\right) d s+\frac{1}{\varepsilon^{1/\alpha}} \int_0^t e^{(t-s)B/\varepsilon} d Z_s
\end{align*}
Then by Minkowski's inequality, we get that for any $0<t\leq T$,
\begin{align*}
\left(\EE|Y^{\varepsilon}_t|^p\right)^{1/p}
\leq&|e^{tB/\vare}y|+\left[\EE\left(\frac{1}{\vare}\int^t_0 \left|e^{\frac{(t-s)B}{\vare}}G(s/\vare,X^{\vare}_s,Y^{\vare}_s)\right|ds\right)^p\right]^{1/p}\nonumber\\
&+\left[\EE\left|\frac{1}{\vare^{1/\alpha}}\int^t_0 e^{(t-s)B/\varepsilon}dZ_s\right|^p\right]^{1/p}\\
\leq&|y|\!+\!\frac{1}{\vare}\int^t_0 e^{\frac{-\beta_1(t-s)}{\vare}}\left[C\!+\!C\left(\EE|X^{\vare}_s|^p\right)^{1/p}\!+\!L_G\left(\EE|Y^{\vare}_s|^p\right)^{1/p}\right]\!+\!C\left(\sum^{\infty}_{k=1}\frac{\gamma^{\alpha}_k }{\beta_k} \right)^{1/\alpha}\\
\leq&C(1+|y|)+\frac{L_{G}}{\beta_1}\sup_{0\leq t\leq T}\left(\EE|Y^{\varepsilon}_t|^p\right)^{1/p}+C\sup_{t\in [0,T]}\left(\mathbb{E}|X_{t}^{\varepsilon} |^p\right)^{1/p}.
\end{align*}
Thus by $L_G<\beta_1$, it follows
\begin{align}
\sup_{t\in [0,T]}\left(\EE|Y^{\varepsilon}_t|^p\right)^{1/p}
\leq C(1+|y|)+C\sup_{t\in [0,T]}\left(\mathbb{E}|X_{t}^{\varepsilon} |^p\right)^{1/p}.\label{F3.3}
\end{align}

Then \eqref{LipFG}, \eqref{rr}, \eqref{LAtheta} and \eqref{F3.3} imply that for any $p\in [1,\alpha)$ and $t\in [0,T]$,
\begin{align*}
\sup_{t\in [0,T]}\left(\mathbb{E}|X_{t}^{\varepsilon} |^p\right)^{1/p}\leq& C|x|\!+\!C\int^T_0 \!\left(\mathbb{E}|X_{t}^{\varepsilon} |^p\right)^{1/p} dt\!+\!C\int^T_0\!\left(\EE|Y^{\varepsilon}_t|^p\right)^{1/p}dt\\
&+\!C\sup_{t\in [0,T]}\left [\mathbb{E}\left|\int^t_0 e^{(t-s)A}dL_s\right|^p\right]^{1/p}\\
\leq &C(1+|x|+|y|)+C\int^T_0 \left(\mathbb{E}|X_{t}^{\varepsilon}|^p\right)^{1/p} dt.
\end{align*}
Then by Gronwall's inequality, we have
\begin{eqnarray*}
\sup_{t\in [0,T]}\left(\mathbb{E}|X_{t}^{\varepsilon} |^p\right)^{1/p}\leq C_{T}(1+|x|+|y|).
\end{eqnarray*}
This together with \eref{F3.3} imply \eref{EY}. The proof is complete.
\end{proof}

The following estimate plays an important role in our proof of main results.
\begin{lemma} \label{L3.9}
For any $T>0$, there exists a constant $C_{T}>0$ such that for any $p\in [1,\alpha)$ and $\delta\in (0,1]$,
\begin{align}
\sup_{\vare>0}\int^{T}_0\left[\mathbb{E}|X^{\varepsilon}_t-X^{\varepsilon}_{t(\delta)}|^p\right]^{1/p}dt\leq C_T\delta^{\frac{\theta }{2}}(1+|x|+|y|),\label{COX}
\end{align}
where $\theta$ is the constant in Assumption \ref{A2} and $t(\delta):=\lfloor t/\delta\rfloor\delta$,  $\lfloor t/\delta\rfloor$ is the integer part of $t/\delta$.
\end{lemma}

\begin{proof}
By \eqref{EX} and Minkowski's inequality, it is easy to check that
\begin{eqnarray}
&&\int^{T}_0\left[\mathbb{E}|X_{t}^{\varepsilon}-X_{t(\delta)}^{\varepsilon}|^p\right]^{1/p}dt\nonumber\\
=\!\!\!\!\!\!\!\!&&\int^{\delta}_0\left(\mathbb{E}|X_{t}^{\varepsilon}-x|^p\right)^{1/p}dt
+\int^{T}_{\delta}\left[\mathbb{E}|X_{t}^{\varepsilon}-X_{t(\delta)}^{\varepsilon}|^p\right]^{1/p}dt\nonumber\\
\leq\!\!\!\!\!\!\!\!&&C_{T}\delta(1+|x|+|y|)+\int^{T}_{\delta}\left[\mathbb{E}|X_{t}^{\varepsilon}-X_{t-\delta}^{\varepsilon}|^p\right]^{1/p}dt+\int^{T}_{\delta}\left[\mathbb{E}|X_{t(\delta)}^{\varepsilon}-X_{t-\delta}^{\varepsilon}|^p\right]^{1/p}dt.\label{FFX1}
\end{eqnarray}
Assumption \ref{A2}, Lemma \ref{L3.1},  \eqref{rr} and \eqref{LAtheta} yield that for any $\theta\in (0,2)$, $t\in(0,T]$.
\begin{align}
&\left(\mathbb{E}\left\|X_{t}^{\varepsilon}\right\|_{\theta}^{p}\right)^{1/p}\nonumber \\
\leq& \left(\mathbb{E}\left\|e^{tA}x\right\|_{\theta}^{p}\right)^{1/p}\! \!+\! \left[\mathbb{E}\left|\int_0^t (-A)^{\theta/2}e^{(t - s)A}F(s/{\varepsilon},X^{\varepsilon}_s, Y^{\varepsilon}_s)ds\right|^p\right]^{1/p}+ \left[\mathbb{E}\left\|\int_0^t e^{(t - s)A}dL_s\right\|_{\theta}^{p}\right]^{1/p} \nonumber \\
&\leq C t^{-\theta/2}|x| + C\int_0^t (t - s)^{-\frac{\theta}{2}}\left(1+\left(\mathbb{E}|X_{s}^{\varepsilon}|^p\right)^{1/p}+\left(\mathbb{E}|Y_{s}^{\varepsilon}|^p\right)^{1/p}\right)ds\!+\! C\left( \sum_{k = 1}^{\infty}\frac{\rho^{\alpha}_k}{\lambda_k^{1 - \alpha\theta/2}}\right)^{1/\alpha} \nonumber \\
&\leq C_T t^{-\theta/2}(1+|x|+|y|).\label{F3.8}
\end{align}

Note that
$$X_{t}^{\varepsilon}-X_{t-\delta}^{\varepsilon}=\left(e^{A\delta}X_{t-\delta}^{\varepsilon}-X_{t-\delta}^{\varepsilon}\right)+\int_{t-\delta}^{t}e^{(t-s)A}F(s/{\varepsilon},X^{\varepsilon}_s, Y^{\varepsilon}_s)ds+\int_{t-\delta}^{t}e^{(t-s)A}dL_s.$$
Minkowski's inequality, Lemma \ref{L3.1}, \eqref{rr} and \eref{F3.8} imply that
\begin{align}  \label{REGX1}
\int^{T}_\delta\left[\mathbb{E}\left|e^{A\delta}X_{t-\delta}^{\varepsilon}-X_{t-\delta}^{\varepsilon}\right|^p\right]^{1/p} dt
\leq&C\delta^{\frac{\theta}{2}}\int^{T}_\delta\left(\mathbb{E}\|X^{\varepsilon}_{t-\delta}\|^p_{\theta}  \right)^{1/p}dt\nonumber\\
\leq&C\delta^{\frac{\theta}{2}}\int^{T}_\delta (t-\delta)^{-\theta/2}(1+|x|+|y|)dt\nonumber\\
\leq&C_T\delta^{\frac{\theta }{2}}(1+|x|+|y|)
\end{align}
and
\begin{align} \label{REGX3}
\int^{T}_\delta \left[\mathbb{E}\left|\int_{t-\delta}^{t}e^{(t-s)A}F(s/{\varepsilon},X^{\varepsilon}_s, Y^{\varepsilon}_s)ds\right|^p\right]^{1/p}dt
\leq&C_{T}(1+|x|+|y|)\delta.
\end{align}
Assumption \ref{A2} and \eqref{LAtheta} imply
\begin{align}
\label{REGX4}
\int_{\delta}^T \left[\mathbb{E}\left|\int_{t -\delta}^{t}e^{(t - s)A}  d L_s \right|^p\right]^{1/p} d t
\leq &C\int_{\delta}^{T}\left[ \sum_{k = 1}^{\infty}\frac{\rho^{\alpha}_k(1 - e^{-\lambda_k \delta})}{\lambda_k}\right]^{1/\alpha} d t \nonumber\\
\leq &C\delta^{\frac{\theta}{2}}\int_{\delta}^{T}\left(\sum_{k = 1}^{\infty}\frac{\rho^{\alpha}_k}{\lambda_k^{1 - \alpha\theta/2}}\right)^{1/\alpha} dt \
\leq C_T\delta^{\frac{\theta}{2}},
\end{align}
where we use the fact that $1-e^{-x}\leq Cx^{\alpha\theta /2}$ for any $x>0$.

Combining \eqref{REGX1}-\eqref{REGX4}, we obtain
\begin{align}\label{FFX2}
\int^{T}_{\delta}\left[\mathbb{E}|X_{t}^{\varepsilon}-X_{t-\delta}^{\varepsilon}|^p\right]^{1/p}dt\leq C_T\delta^{\frac{\theta }{2}}(1+|x|+|y|).
\end{align}
Similarly, we have
\begin{align}\label{FFX3}
\int^{T}_{\delta}\left[\mathbb{E}|X_{t(\delta)}^{\varepsilon}-X_{t-\delta}^{\varepsilon}|^p\right]^{1/p}dt\leq C_T\delta^{\frac{\theta }{2}}(1+|x|+|y|).
\end{align}

Finally, \eref{FFX1}, (\ref{FFX2}) and (\ref{FFX3}) imply \eref{COX}. The proof is complete.
\end{proof}

\subsection{Auxiliary process}
Inspired from the idea of Khasminskii in \cite{K1968}, we construct an auxiliary process $\hat{Y}_{t}^{\varepsilon}$,
i.e., we split $[0,T)$ into some subintervals of size $\delta>0$, where $\delta$ depends on $\vare$ and will be chosen later.
We construct the process $\hat{Y}_{t}^{\varepsilon}$ on each time interval $[k\delta,(k+1)\delta\wedge T)$, $k=0,1,\ldots, \lfloor T/\delta \rfloor $, as follows:
$$
d\hat{Y}_{t}^{\varepsilon}=\frac{1}{\varepsilon} [ B \hat{Y}^\varepsilon_t+G\left(t/\varepsilon,X_{k\delta}^{\varepsilon},\hat{Y}_{t}^{\varepsilon}) \right] dt + \frac{1}{\varepsilon^{1/\alpha}}dZ_t, t\in[k\delta,(k+1)\delta\wedge T), \quad \hat{Y}_{k\delta}^{\varepsilon}=Y^{\vare}_{k\delta},
$$
which has an unique mild solution
\begin{align} \label{AuxiliaryPro Y 01}
\hat{Y}_{t}^{\varepsilon}=e^{(t-k\delta)B/\varepsilon}Y^{\vare}_{k\delta}+\frac{1}{\varepsilon} \int_{k\delta}^t e^{(t-s)B/\varepsilon} G\left(s / \varepsilon, X_{k\delta}^\varepsilon, \hat Y_s^\varepsilon\right) d s+\frac{1}{\varepsilon^{1/\alpha}} \int_{k\delta}^t e^{(t-s)B/\varepsilon} d Z_s.
\end{align}

Then we have the following result.

\begin{lemma}
For any $T>0$, there exists a constant $C_T>0$ such that for any $p\in [1,\alpha)$ and $\delta\in (0,1]$, we have
\begin{align}
&\sup_{\vare>0,t\in [0,T]}\left(\mathbb{E}|\hat{Y}_{t}^{\varepsilon}|^{p}\right)^{1/p}\leq C_{T}(1+|x|+|y|),\label{EAY}\\
&\sup_{\vare>0}\int^{T}_0\left[\EE|Y_{t}^{\varepsilon}-\hat Y_{t}^{\varepsilon}|^p\right]^{1/p} dt
\leq C_T\delta^{\frac{\theta }{2}}(1+|x|+|y|).\label{DEY}
\end{align}
\end{lemma}

\begin{proof}
As the proof of \eref{EY}, we can easily obtain \eqref{EAY}. Next, we prove \eqref{DEY}.
The definition of $Y_{t}^{\varepsilon}$ and $\hat{Y}_{t}^{\varepsilon}$, as well as \eref{LipFG} imply that for any $t\in [k\delta,(k+1)\delta\wedge T)$,
\begin{align*}
|Y_{t}^{\varepsilon}-\hat Y_{t}^{\varepsilon}|=&\frac{1}{\vare}\left|\int^t_{k\delta} e^{(t-s)B/\varepsilon}[G\left(s / \varepsilon, X_s^\varepsilon, Y_s^\varepsilon\right)-G(s / \varepsilon,X_{k\delta}^{\varepsilon},\hat{Y}_{s}^{\varepsilon})]ds\right|\\
\leq & \frac{C}{\vare}\int^t_{k\delta} e^{-\beta_1(t-s)/\vare}|X_{s}^{\varepsilon}-X_{s(\delta)}^{\varepsilon}|ds+\frac{L_G}{\vare}\int^t_{k\delta} e^{-\beta_1(t-s)/\vare}|Y_{s}^{\varepsilon}-\hat Y_{s}^{\varepsilon}|ds.
\end{align*}
By integrating both sides of the above inequality with respect to time from $k\delta$ to $(k+1)\delta\wedge T$, and then summing over $k$, we can obtain
\begin{align*}
&\int^{T}_0|Y_{t}^{\varepsilon}-\hat Y_{t}^{\varepsilon}|dt=\sum^{\lfloor T/\delta\rfloor }_{k=0}\int^{(k+1)\delta \wedge T}_{k\delta}|Y_{t}^{\varepsilon}-\hat Y_{t}^{\varepsilon}|dt\\
\leq&\frac{C}{\vare}\!\sum^{\lfloor T/\delta\rfloor }_{k=0}\!\int^{(k+1)\delta \wedge T}_{k\delta}\!\!\int^t_{k\delta} e^{-\beta_1(t-s)/\vare}|X_{s}^{\varepsilon}\!-\!X_{s(\delta)}^{\varepsilon}|dsdt\!+\!\frac{L_G}{\vare}\sum^{\lfloor T/\delta\rfloor }_{k=0}\int^{(k+1)\delta \wedge T}_{k\delta}\!\!\int^t_{k\delta}  e^{-\beta_1(t-s)/\vare}|Y_{s}^{\varepsilon}\!-\!\hat Y_{s}^{\varepsilon}|dsdt\\
\leq&\frac{C}{\vare}\sum^{\lfloor T/\delta\rfloor }_{k=0}\int^{(k+1)\delta \wedge T}_{k\delta}\!\!\left(\int^{(k+1)\delta \wedge T}_{s} e^{-\beta_1(t-s)/\vare}|X_{s}^{\varepsilon}\!-\!X_{s(\delta)}^{\varepsilon}|dt\right)ds\\
&+\frac{L_G}{\vare}\sum^{\lfloor T/\delta\rfloor }_{k=0}\int^{(k+1)\delta \wedge T}_{k\delta}\!\!\left(\int^{(k+1)\delta \wedge T}_{s}   e^{-\beta_1(t-s)/\vare}|Y_{s}^{\varepsilon}\!-\!\hat Y_{s}^{\varepsilon}|dt\right)ds\\
\leq &\frac{C}{\beta_1}\int^{T}_0|X_{s}^{\varepsilon}\!-\!X_{s(\delta)}^{\varepsilon}|ds\!+\!\frac{L_G}{\beta_1}\int^{T}_0|Y_{s}^{\varepsilon}\!-\!\hat Y_{s}^{\varepsilon}|ds,
\end{align*}
where Fubini's theorem is applied in the second inequality. Then by  Minkowski's inequality, Lemma \ref{L3.9} and $L_G<\beta_1$, it is easy to see
\begin{align*}
\int^{T}_0\left[\EE|Y_{t}^{\varepsilon}-\hat Y_{t}^{\varepsilon}|^p\right]^{1/p} dt
\leq \frac{C}{\beta_1}\int^{T}_{0}\left[\mathbb{E}|X_{t}^{\varepsilon}-X_{t(\delta)}^{\varepsilon}|^p\right]^{1/p}dt\leq C_T\delta^{\frac{\theta }{2}}(1+|x|+|y|).
\end{align*}
The proof is complete.
\end{proof}

\section{Evolution system of measures for SPDE driven by $\alpha$-stable process}\label{Section4}

In order to study the evolution system of measures for time-inhomogeneous frozen equation, we need to extend the time domain of the noise to the whole line. Thus we take another cylindrical $\alpha$-stable process $\tilde{Z}_t=\sum^{\infty}_{k=1}\gamma_{k}\tilde{Z}^{k}_{t}e_k$,
where $\{\tilde{Z}^{k}\}_{k\geq 1}$ is an independent copy of $\{Z^{k}\}_{k\geq 1}$. We define
\begin{equation*}
Z_t=\begin{cases}
Z_t,\quad\,\, t\geq 0,\\
\tilde{Z}_{-t}, \quad t< 0,
\end{cases}
\end{equation*}
which is a cylindrical $\alpha$-stable process on $\RR$.

Now, we consider the following time-inhomogeneous frozen SDE on $\RR$:
\begin{eqnarray}\label{FZE}
dY_{t}=\left[BY_{t}+G(t,x,Y_{t})\right]dt+d Z_{t},\quad Y_{s}=y\in H.
\end{eqnarray}
Since $G(t,x,\cdot)$ is Lipshcitz continuous, it is easy to prove that equation \eref{FZE} has a unique mild solution denoted by $\left\{Y_t^{s, x, y}\right\}_{t \geq s}$, that is
\begin{align}
Y^{s,x,y}_t=e^{(t-s)B}y+\int_s^t e^{(t-r)B}G\left(r, x, Y^{s,x,y}_r\right) d r+ \int_s^t e^{(t-r)B} d Z_r.\label{MildFY}
\end{align}

The following estimates also hold:

\begin{lemma}
There exists a constant $C>0$ such that for any $p\in [1,\alpha)$ and $x,y\in H $,
\begin{align}
\sup_{t\geq s}\left(\mathbb{E}|Y_{t}^{s,x,y}|^{p}\right)^{1/p}\leq C(1+|x|+|y|).\label{EFY}
\end{align}
Moreover, for any $x_i, y_i \in H $, $i=1,2$, we have
\begin{align}
\left|Y_t^{s, x_1, y_1}-Y_t^{s, x_2, y_2}\right| \leq e^{-\frac{\left(\beta_1-L_G\right)(t-s)}{2}}\left|y_1-y_2\right|+C\left|x_1-x_2\right| .\label{IncFY}
\end{align}
\end{lemma}

\begin{proof}
(1) Recall that \eref{MildFY} and by Minkowski's inequality, we get for any $t\geq s$,
\begin{align*}
\left(\EE|Y_{t}^{s,x,y}|^p\right)^{1/p}\!
\leq&|e^{(t-s)B}y|\!+\!\left[\EE\left(\int^t_s\!\!|e^{(t-r)B}G(r,x,Y_{r}^{s,x,y})|dr\right)^p\right]^{1/p}\!+\!\left[\EE\left|\int^t_s\!\! e^{(t-r)B}dZ_r\right|^p\right]^{1/p}\\
\leq&|y|\!+\!\!\int^t_s\!\!e^{-\beta_1(t-r)}\left[C\!+\!C|x|\!+\!L_G\left(\EE|Y_{r}^{s,x,y}|^p\right)^{1/p}\right]dr\!+\!C\left(\sum^{\infty}_{k=1}\frac{\gamma^{\alpha}_k }{\beta_k} \right)^{1/\alpha}\\
\leq&C(1+|x|+|y|)+\frac{L_{G}}{\beta_1}\sup_{t\geq s}\left(\EE|Y_{t}^{s,x,y}|^p\right)^{1/p}.
\end{align*}
Note that $L_G<\beta_1$, which yields that
\begin{align*}
\sup_{t\geq s}\left(\EE|Y_{t}^{s,x,y}|^p\right)^{1/p}\leq C(1+|x|+|y|).
\end{align*}

(2) Let $N_t=: Y_t^{s, x_1, y_1}-Y_t^{s, x_2, y_2}$.  It is easy to see
$$
\left\{\begin{array}{l}
d N_t=\left[B N_t+G\left(t, x_1, Y_t^{s, x_1, y_1}\right)-G\left(t, x_2, Y_t^{s, x_2, y_2}\right)\right] d t, \\
N_s=y_1-y_2 .
\end{array}\right.
$$
By Young's inequality, we obtain
\begin{align*}
\frac{d}{dt} \left| N_t \right|^2 & = \left\langle B N_t, 2 N_t \right\rangle + \left\langle G\left(t, x_1, Y_t^{s, x_1, y_1}\right) - G\left(t, x_2, Y_t^{s, x_2, y_2}\right), 2 N_t \right\rangle \\
& \leq -\left(\beta_1 - L_G\right) \left| N_t \right|^2 + C \left| x_1 - x_2 \right|^2.
\end{align*}
According to the comparison theorem, the following inequality holds
$$
\left|N_t\right|^2 \leq e^{-\left(\beta_1-L_G\right)(t-s)}\left|y_1-y_2\right|^2+C\left|x_1-x_2\right|^2,
$$
which implies \eref{IncFY} holds. The proof is complete.
\end{proof}

\begin{lemma}
For any $t \in \mathbb{R}$, $x,y \in H$, there exists $\eta_t^x \in L^p(\Omega,H)$ (independent of $y$) such that for any $p\in [1,\alpha)$ and $t \geq s $
\begin{align}
 \left[\mathbb{E}\left|Y_t^{s, x, y}-\eta_t^x\right|^p\right]^{1/p} \leq C e^{-\frac{\left(\beta_1-L_G\right)(t-s)}{2}}(1+|x|+|y|), \quad \sup_{t\in \RR}\EE|\eta_t^x|^p\leq C_p(1+|x|^p).
\label{b}
\end{align}
\end{lemma}

\begin{proof}
Note that for any $h>0$,
\begin{align*}
Y^{s-h,x,y}_t=e^{(t-s)B}Y^{s-h,x,y}_s+\int_s^t e^{(t-r)B}G\left(r, x, Y^{s,x,y}_r\right) d r+ \int_s^t e^{(t-r)B} d Z_r.
\end{align*}
Let $H^{s,h,x,y}_t=: Y^{s,x,y}_t-Y^{s-h,x,y}_t$. Then it is easy to see that
$$
\left\{\begin{array}{l}
d H^{s,h,x,y}_t=\left[B H^{s,h,x,y}_t+G\left(t, x, Y^{s,x,y}_t\right)-G\left(t, x,Y^{s-h,x,y}_t\right)\right] d t, \\
H^{s,h,x,y}_s=y-Y^{s-h,x,y}_s .
\end{array}\right.
$$
Then by the same argument above, we have
$$
\left|H^{s,h,x,y}_t\right|\leq e^{-\frac{\left(\beta_1-L_G\right)(t-s)}{2}}\left|y-Y^{s-h,x,y}_s\right|.
$$
Hence, by \eref{EFY}, it follows
\begin{align*}
\left[\EE\left|H^{s,h,x,y}_t\right|^p\right]^{1/p}\leq & e^{-\frac{\left(\beta_1-L_G\right)(t-s)}{2}}\left[\EE\left|y-Y^{s-h,x,y}_s\right|^p\right]^{1/p}\\
\leq & C e^{-\frac{\left(\beta_1-L_G\right)(t-s)}{2}}(1+|x|+|y|).
\end{align*}
 This implies for any $ t\in \mathbb{R}$ and $ x,y\in H$, $\{Y^{s,x,y}_t\}_{s\leq t}$ is a Cauchy sequence in $L^{p} (\Omega , H )$ as $s\to -\infty$, then there exists an element  $\eta^{x,y}_t\in L^{p} (\Omega , H )$ such that
\begin{align}
\left[\EE\left|Y^{s,x,y}_t-\eta^{x,y}_t\right|^p\right]^{1/p}
\leq C e^{-\frac{\left(\beta_1-L_G\right)(t-s)}{2}}(1+|x|+|y|).\label{F4.6}
\end{align}

Next, we are going to prove $\eta^{x,y}_{t}$ is independent of $y$. In fact, \eref{IncFY} implies that for any $y_1,y_2\in H$, we have
\begin{align*}
\lim_{s\to-\infty}\left[\EE\left|Y_t^{s, x, y_1}-Y_t^{s, x, y_2}\right|^p\right]^{1/p}=0.
\end{align*}
Hence, it follows
\begin{align*}
&\left[\EE\left|\eta^{x,y_1}_{t}-\eta^{x,y_2}_{t}\right|^p\right]^{1/p}\\
\leq & \left[\EE\left|\eta^{x,y_1}_{t}-Y^{s,x,y_1}_t\right|^p\right]^{1/p}+\left[\EE\left|\eta^{x,y_2}_{t}-Y^{s,x,y_2}_t\right|^p\right]^{1/p}+\left[\EE\left|Y^{s,x,y_1}_t-Y^{s,x,y_2}_t\right|^p\right]^{1/p}\\
\leq& C e^{-(\beta_1-L_G)(t-s)/2}(1+|x|+|y_1|+|y_2|).
\end{align*}
Taking $s\rightarrow -\infty$, we get $\eta^{x,y_1}_{t}=\eta^{x,y_2}_{t}$ in $L^{p} (\Omega , H )$. Thus the first estimate is proved in  \eqref{b}.

The second estimate in \eqref{b} follows from \eqref{F4.6} and \eqref{EFY}. The proof is complete.
\end{proof}

Note that $\left\{Y_t^{s, x, y}\right\}_{t \geq s}$ is a time-inhomogeneous Markov process,  denote $P^x_{s,t}$ is the corresponding transition semigroup of $\left\{Y_t^{s, x, y}\right\}_{t \geq s}$.
Inspired from \cite{DR2006}, we give the following definition.

\begin{definition}
A class of measures $\{\nu^x_t\}_{t\in\RR}$ is called  an evolution system of measures for $\{P^x_{s,t}\}_{t\geq s}$ if for any $s\leq t,\varphi\in C_b(H)$, it follows
\begin{align}
\int_{H}P^x_{s,t} \varphi(y)\nu^x_s(dy)=\int_{H}\varphi(y)\nu^x_t(dy).\label{ESM}
\end{align}
\end{definition}

For any $t\in \mathbb{R}$ and $x\in H$, we denote by $\mu_t^x$ as the law of $\eta^x_t$. The following result holds, since its proof closely follows the arguments in \cite{DR2006}, we omit the detailed proof.

\begin{proposition}\label{pp}
The class $\left\{\mu_t^x\right\}_{t \in \mathbb{R}}$ is a solution of an evolution system of measures indexed by $\RR$ for $\{P^x_{s,t}\}_{t\geq s}$.  Moreover, for any Lipschitz continuous function $\phi$ on $H$,
\begin{align}\label{Wergodicity}
 \left|P_{s, t}^x \phi(y)-\int_H\phi(z) \mu_t^x(d z)\right| \leq C \operatorname{Lip}(\phi)(1+|x|+|y|)e^{-(t-s)\left(\beta_1-L_G\right)/2},
\end{align}
where ${\rm Lip}(\phi):=\sup _{x \neq y} |\phi(x)-\phi(y)|/|x-y|$.
Furthermore, if $\left\{\nu_t^x\right\}_{t \in \mathbb{R}}$ is another evolution system of measures for $\{P^x_{s,t}\}_{t\geq s}$ satisfying
$$
\sup_{t \in \mathbb{R}} \int_H |y| \, \nu_t^x(dy) < \infty, \quad \forall x \in H,
$$
then $\nu_t^x = \mu_t^x$ for all $t \in \mathbb{R}$ and $x \in H$.
\end{proposition}

\section{Proofs of main results}

In this section, we will give the detailed proofs of Theorems \ref{MR1}, \ref{MR2} and \ref{MR3}.

\subsection{Proof of Theorem 2.3}

We begin by establishing the well-posedness of the solution for the averaged equation \eref{AV1}.

\begin{lemma} For any $x\in H$ and $\varepsilon>0$,  Eq.\eqref{AV1} has a unique mild solution $\{\bar{X}_t^\varepsilon\}_{t\geq 0}$, which satisfies
$$
\bar{X}_t^\varepsilon = e^{t A} x + \int_0^t e^{(t - r) A} \bar{F}\left(r / \varepsilon, \bar{X}_r^\varepsilon\right) dr + \int_0^t e^{(t - r) A} d L_r,\quad t\geq 0.
$$
Moreover, for any $T>0$, there exists a constant $C_{T}>0$ such that for any $p\in [1,\alpha)$,
\begin{align}
\sup_{\vare>0,t\in [0,T] }\left(\mathbb{E}|\bar X_{t}^{\vare}|^{p}\right)^{1/p}\leq C_{T}(1+|x|)\label{EbarX}
\end{align}
and for any $\delta\in (0,1]$,
\begin{align}
\sup_{\vare>0}\int^{T}_0\left[\mathbb{E}|\bar X^{\varepsilon}_t-\bar X^{\varepsilon}_{t(\delta)}|^p\right]^{1/p}dt\leq C_T\delta^{\frac{\theta }{2}}(1+|x|).\label{CObarX}
\end{align}
\end{lemma}

\begin{proof}
It suffices to verify that for any $x_1, x_2 \in H, t\geq 0$.
\begin{align}
|\bar F(t,x_{1})-\bar F(t,x_{2})|\le  C|x_{1}-x_{2}|,\quad\quad |\bar{F}(t,x_1)|\leq C(1+|x_1|).\label{barc1}
\end{align}
Then Eq.\eqref{AV1} has a unique mild solution $\{\bar{X}_t^\varepsilon\}_{t\geq 0}$. Moreover, the proofs of the estimates \eqref{EbarX} and \eqref{CObarX} follow the same argument as used in Lemmas \ref{L3.1} and \ref{L3.9}.

Indeed, using \eqref{IncFY} and \eqref{Wergodicity}, we obtain for any $y\in H$,
\begin{align*}
& \left| \bar{F}\left(t, x_1\right) - \bar{F}\left(t, x_2\right) \right| \\
\leq & \left| \bar{F}\left(t, x_1\right) - \mathbb{E} F\left(t, x_1, Y_t^{s, x_1, \textbf{0}}\right) \right| + \left| \mathbb{E} F\left(t, x_2, Y_t^{s, x_2, \textbf{0}}\right) - \bar{F}\left(t, x_2\right) \right| \\
& + \mathbb{E} \left| F\left(t, x_1, Y_t^{s, x_1, \textbf{0}}\right) - F\left(t, x_2, Y_t^{s, x_2, \textbf{0}}\right) \right| \\
\leq & Ce^{-\left(\beta_1-L_G\right)(t-s)/2}\left(1 + \left| x_1 \right| + \left| x_2 \right|\right) \\
& + C \left[ \left| x_1 - x_2 \right| + \mathbb{E} \left| Y_t^{s, x_1, \textbf{0}} - Y_t^{s, x_2, \textbf{0}} \right| \right] \\
\leq & Ce^{-\left(\beta_1-L_G\right)(t-s)/2} \left(1 + \left| x_1 \right| + \left| x_2 \right|\right) + C \left| x_1 - x_2 \right|,
\end{align*}
where $\textbf{0}$ is the zero element in $H$. By letting $ s \rightarrow -\infty $, we obtain
$$
\left| \bar{F}\left(t, x_1\right) - \bar{F}\left(t, x_2\right) \right| \leq C \left| x_1 - x_2 \right|.
$$

Meanwhile, using \eqref{LipFG} and \eqref{b}, we get
\begin{align*}
|\bar F(t,x_{1})|
=\left|\int_{H} F(t,x_{1},y)\mu^{x_1}_t(dy)\right|\le  C\int_{H} (1+|x_1|+|y|)\mu^{x_1}_t(dy)\le C(1+|x_1|).
\end{align*}
The proof is complete.
\end{proof}

\vspace{0.2cm}

In the following, we provide a detailed proof of Theorem \ref{MR1}.

\begin{proof}[Proof of Theorem $\ref{MR1}$] We will divide the proof into three steps.

\textbf{Step 1}: Denote $Z^{\varepsilon}_t:=X_t^\varepsilon - \bar{X}_t^\varepsilon $,  it is easy to check
\begin{align}\label{Zepsilon}
d Z^{\varepsilon}_t=& AZ^{\varepsilon}_t dt+\left[F(t/\varepsilon,X_{s}^\varepsilon,Y_s^\varepsilon)-\bar{F}(t/\varepsilon,\bar{X}^{\varepsilon}_t)\right]dt,\quad Z^{\varepsilon}_0=\textbf{0}\in H.
\end{align}

Fix $p\in (1,\alpha)$, we construct $U_{\rho}(x):=(|x|^2+\rho)^{p/2}$ for $\rho>0$, then
$$
\partial_x U_{\rho}(x)= p x (|x|^2+\rho)^{p/2-1}, \quad \partial^2_x U_{\rho}(x)= p I (|x|^2+\rho)^{p/2-1}+ p(p-2) x\otimes x(|x|^2+\rho)^{p/2-2},
$$
where $\|\cdot\|$ is the usual operator norm. As a result, we have
$$
|\partial_x U_{\rho}(x)|\leq p|x|^{p-1},\quad \|\partial^2_x U_{\rho}(x)\|\leq p(3-p)|x|^{p-2}.
$$
Next, we shall prove $\partial_x U_{\rho}(x)$ is $(p-1)$-H\"{o}lder continuous, that is
\begin{align}
|\partial_x U_{\rho}(x_1)-\partial_x U_{\rho}(x_2)|\leq C_p|x_1-x_2|^{p-1},\quad x_1,x_2\in H.\label{F5.5}
\end{align}
Indeed, without loss of generality, we assume that $|x_1|\geq |x_2|$. Then
\begin{align*}
|\partial_x U_{\rho}(x_1) -\partial_x U_{\rho}(x_2)| \leq& \int_{0}^{1} \|\partial^2_x U(x_2 + \xi(x_1 - x_2))\| |x_1- x_2| d\xi\\
\leq & p(3-p) |x_1- x_2| \int_{0}^{1} |x_2 + \xi(x_1 - x_2)|^{p - 2} d\xi.
\end{align*}

If $|x_1 - x_2| \leq |x_1| / 2 $, note that
$$
|x_2 + \xi(x_1 - x_2)| \geq |x_1| - \xi |x_1 - x_2|  \geq |x_1-x_2|,\quad \forall \xi\in[0,1].
$$
Thus we get
\begin{align}
|\partial_x U_{\rho}(x_1) -\partial_x U_{\rho}(x_2)| \leq p(3-p) |x_1- x_2|^{p-1}.\label{F5.6}
\end{align}

On the other hand, if $|x_1-x_2| \geq |x_1| / 2 $, then
\begin{align}
|\partial_x U_{\rho}(x_1)-\partial_x U_{\rho}(x_2)|
\leq p|x_1|^{p-1}+ p|x_2|^{p-1} \leq 2 p |x_1|^{p - 1}
\leq 2^{p}p |x_1 - x_2|^{p-1}.\label{F5.7}
\end{align}
Therefore, \eqref{F5.6} and \eqref{F5.7} imply \eqref{F5.5} holds.

Multiply both sides of Eq \eqref{Zepsilon} by $\partial_xU_{\rho}(Z^{\varepsilon}_t)$, then integrate from $0$ to $T$, and finally take the mathematical expectation, we obtain that for any $t\in [0,T]$,
\begin{align*}
\EE U_{\rho}(Z^{\varepsilon}_t)=&U_{\rho}(\textbf{0})+\EE\int^t_0 \langle AZ^{\varepsilon}_s, \partial_x U_{\rho}(Z^{\varepsilon}_s)\rangle ds\\
&+\EE\int^t_0 \langle F(s/\varepsilon,X_{s}^\varepsilon,Y_s^\varepsilon)-\bar{F}(s/\varepsilon, X^{\varepsilon}_s), \partial_x U_{\rho}(Z^{\varepsilon}_s)\rangle ds\\
&+\EE\int^t_0 \langle \bar{F}(s/\varepsilon, X^{\varepsilon}_s)-\bar{F}(s/\varepsilon,\bar{X}^{\varepsilon}_s), \partial_x U_{\rho}(Z^{\varepsilon}_s)\rangle ds\\
\leq &U_{\rho}(\textbf{0})\!+\!\EE\int^t_0\! \langle F(s/\varepsilon,X_{s}^\varepsilon,Y_s^\varepsilon)\!-\!\bar{F}(s/\varepsilon, X^{\varepsilon}_s), \partial_x U_{\rho}(Z^{\varepsilon}_s)\rangle ds\!+\!C_p\int^t_0 \EE U_{\rho}(Z^{\varepsilon}_s) ds,
\end{align*}
where we use by \eqref{barc1} in the last inequality. Then it follows
\begin{align*}
\sup_{t\in [0,T]}\EE U_{\rho}(Z^{\varepsilon}_t)\leq &\rho^{p/2}+\sup_{t\in [0,T]}\left|\EE\int^t_0 \langle F(s/\varepsilon,X_{s}^\varepsilon,Y_s^\varepsilon)-\bar{F}(s/\varepsilon,X^{\varepsilon}_s), \partial_x U_{\rho}(Z^{\varepsilon}_s)\rangle ds\right|\\
&+C_p\int^T_0 \EE U_{\rho}(Z^{\varepsilon}_s) ds.
\end{align*}
By Gronwall inequality, we get for any $\delta\in(0,1]$,
\begin{align*}
&\sup_{t\in [0,T]}\EE U_{\rho}(Z^{\varepsilon}_t)\\
\leq &C_{p,T}\rho^{p/2}+C_{p,T}\sup_{t\in [0,T]}\left|\EE\int^t_0 \langle F(s/\varepsilon,X_{s}^\varepsilon,Y_s^\varepsilon)-\bar{F}(s/\varepsilon,X^{\varepsilon}_s), \partial_x U_{\rho}(Z^{\varepsilon}_s)\rangle ds\right|\\
\leq& C_{p,T}\rho^{p/2}+C_{p,T}\sup_{t\in [0,T]}\left|\EE\int^t_0 \langle F(s/\varepsilon,X_{s}^\varepsilon,Y_s^\varepsilon)-\bar{F}(s/\varepsilon,X^{\varepsilon}_s), \partial_x U_{\rho}(Z^{\varepsilon}_s)-\partial_x U_{\rho}(Z^{\varepsilon}_{s(\delta})\rangle ds\right|\\
&+C_{p,T}\sup_{t\in [0,T]}\Big|\EE\int^t_0 \langle F(s/\varepsilon,X_{s}^\varepsilon,Y_s^\varepsilon)-F(s/\varepsilon,X_{s(\delta)}^\varepsilon,\hat Y_s^\varepsilon),\partial_x U_{\rho}(Z^{\varepsilon}_{s(\delta)})\rangle ds\\
&\quad\quad\quad\quad\quad\quad\quad-\EE\int^t_0 \langle\bar{F}(s/\varepsilon,X^{\varepsilon}_s)-\bar{F}(s/\varepsilon,X^{\varepsilon}_{s(\delta)}),\partial_x U_{\rho}(Z^{\varepsilon}_{s(\delta)})\rangle ds\Big|\\
&+C_{p,T}\sup_{t\in [0,T]}\left|\EE\int^t_0 \langle F(s/\varepsilon,X_{s(\delta)}^\varepsilon,\hat Y_s^\varepsilon)-\bar{F}(s/\varepsilon,X^{\varepsilon}_{s(\delta)}), \partial_x U_{\rho}(Z^{\varepsilon}_{s(\delta)})\rangle ds\right|\\
=:&C_{p,T}\rho^{p/2}+I^{\vare}_1(T)+I^{\vare}_2(T)+I^{\vare}_3(T).
\end{align*}

\textbf{Step 2}: In this step, we indent to prove terms $I^{\vare}_1(T)$ and $I^{\vare}_2(T)$. For $I^{\vare}_1(T)$, using \eqref{F5.5}, \eqref{COX} and \eqref{CObarX}, we can obtain
\begin{align}
I^{\vare}_1(T) \leq &C_{p,T} \int_0^T\mathbb{E}\left[(1+|X_{s}^\varepsilon|+|Y_s^\varepsilon|)\left|Z_s^{\varepsilon}-Z_{s(\delta)}^{\varepsilon}\right|^{p-1}\right]ds \nonumber\\
\leq &C_{p,T} \int_0^T\left[\mathbb{E}(1+|X_{s}^\varepsilon|^p+|Y_s^\varepsilon|^p)\right]^{1/p}\left[\EE\left|Z_s^{\varepsilon}-Z_{s(\delta)}^{\varepsilon}\right|^{p}\right]^{\frac{p-1}{p}}ds \nonumber\\
\leq &C_{p,T}(1+|x|+|y|) \left\{\int_0^T\left[\EE\left|X_s^{\varepsilon}-X_{s(\delta)}^{\varepsilon}\right|^{p}\right]^{\frac{1}{p}}
+\left[\EE\left|\bar X_s^{\varepsilon}-\bar X_{s(\delta)}^{\varepsilon}\right|^{p}\right]^{\frac{1}{p}}ds\right\}^{p-1} \nonumber\\
\leq &C_{p,T} \delta^{\theta(p-1)/2}\left(1+|x|^p+|y|^p\right).\label{I1}
\end{align}

For $I^{\vare}_2(t)$, according to \eqref{LipFG}, \eqref{barc1}, \eqref{EX} and \eqref{EbarX}, we have
\begin{align}
I^{\vare}_2(T) \leq &C_{p,T} \int_0^T \mathbb{E}\left[\left(\left|X_s^{\varepsilon}-X_{s(\delta)}^{\varepsilon}\right|+\left|Y_s^{\varepsilon}-\hat{Y}_s^{\varepsilon}\right|\right)
|Z^{\varepsilon}_{s(\delta)}|^{p-1}\right]ds\nonumber\\
\leq &C_{p,T} \int_0^T \left\{\left[\mathbb{E}\left|X_s^{\varepsilon}-X_{s(\delta)}^{\varepsilon}\right|^p\right]^{1 / p}+\left[\mathbb{E}\left|Y_s^{\varepsilon}-\hat{Y}_s^{\varepsilon}\right|^p\right]^{1 /p}\right\} \left[\EE|Z^{\varepsilon}_{s(\delta)}|^p\right]^{\frac{p-1}{p}}ds \nonumber\\
\leq& C_{p,T} \delta^{\theta/2}\left(1+|x|^p+|y|^p\right).\label{I2}
\end{align}

Consequently, if we can prove
\begin{align}
I^{\vare}_3(T) \leq C_{p,T}(1+|x|^p+|y|^p)\left(\delta+\varepsilon/\delta\right).\label{I3}
\end{align}
Then together with \eref{I1}-\eref{I3}, we have
\begin{align*}
\sup_{t\in [0,T]}\EE U_{\rho}(Z^{\varepsilon}_t)
\leq C_{p,T}\rho^{p/2}+C_{p,T}(1+|x|^p+|y|^p)\left(\varepsilon/\delta+\delta^{\theta(p-1)/2}\right).
\end{align*}
Takeing $\delta=\varepsilon^{\frac{2}{\theta(p-1)+2}}$, we obtain
\begin{align*}
 \sup _{t \in[0, T]}\mathbb{E}\left|X_t^{\varepsilon}-\bar{X}_t^{\varepsilon}\right|^p \leq &\sup_{t\in [0,T]}\EE U_{\rho}(Z^{\varepsilon}_t)
\leq C_{p,T}\rho^{p/2}+C_{p,T}(1+|x|^p+|y|^p)\varepsilon^{\frac{\theta(p-1)}{\theta(p-1)+2}}.
\end{align*}
Finally, \eqref{R1} holds by letting $\rho\to 0$.

\textbf{Step 3}: In this step, we prove \eqref{I3}. Note that
\begin{align*}
I^{\vare}_3(T) \leq & C_{p,T}\sup_{t\in [0,T]}\left| \sum_{k = 0}^{\lfloor t / \delta \rfloor - 1} \int_{k \delta}^{(k + 1) \delta} \EE\left\langle F\left(s/{\varepsilon},X_{k \delta}^{\varepsilon},  \hat{Y}_s^{\varepsilon}\right)-\bar{F}\left(s/{\varepsilon},X_{k \delta}^{\varepsilon}\right), \partial_x U_{\rho}(Z^{\varepsilon}_{k\delta})\right\rangle ds\right|\\
&+C_{p,T}\sup_{t\in [0,T]}\left|\int_{\lfloor t / \delta \rfloor \delta}^t \mathbb{E}\left\langle F\left(s/\varepsilon,X_{s(\delta)}^{\varepsilon},  \hat{Y}_s^{\varepsilon}\right)-\bar{F}\left(s/\varepsilon,X_{s(\delta)}^{\varepsilon}\right), \partial_x U_{\rho}(Z^{\varepsilon}_{s(\delta)})\right\rangle ds\right|\\
\leq & C_{p,T}\sum_{k = 0}^{\lfloor T / \delta \rfloor - 1} \left| \int_{k \delta}^{(k + 1) \delta} \EE\left\langle F\left(s/{\varepsilon},X_{k \delta}^{\varepsilon},  \hat{Y}_s^{\varepsilon}\right)-\bar{F}\left(s/{\varepsilon},X_{k \delta}^{\varepsilon}\right), \partial_x U_{\rho}(Z^{\varepsilon}_{k\delta})\right\rangle ds\right|\\
&+C_{p,T}\sup_{t\in [0,T]}\left|\int_{\lfloor t / \delta \rfloor \delta}^t \mathbb{E}\left\langle F\left(s/\varepsilon,X_{s(\delta)}^{\varepsilon},  \hat{Y}_s^{\varepsilon}\right)-\bar{F}\left(s/\varepsilon,X_{s(\delta)}^{\varepsilon}\right), \partial_x U_{\rho}(Z^{\varepsilon}_{s(\delta)})\right\rangle ds\right|\\
=:& I^{\vare}_{31}(T)+I^{\vare}_{32}(T).
\end{align*}

For $I^{\vare}_{32}(T)$, it is easy to know
\begin{align*}
 I^{\vare}_{32}(T) \leq  C_{p,T} \delta \sup_{s\in [0,T]}\mathbb{E}\left[(1+|X_{s(\delta)}^\varepsilon|+|\hat Y_s^\varepsilon|)\left|Z_{s(\delta)}^{\varepsilon}\right|^{p-1}\right]\leq  C_{p,T}(1+|x|^p+|y|^p) \delta.
\end{align*}

Next, we indent to deal with term $I^{\vare}_{31}(T)$.  For any $s \leq t$ and random variables $X, Y \in \mathscr{F}_s$, we consider the following equation:
$$
d \tilde{Y}_t^{\varepsilon, s, X, Y}=\frac{1}{\varepsilon}\left[B\tilde{Y}_t^{\varepsilon, s, X, Y}+G\left(t / \varepsilon, X, \tilde{Y}_t^{\varepsilon, s, X, Y}\right)\right] d t+\frac{1}{\varepsilon^{1/{\alpha}}}d Z_t, \quad t \geq s
$$
with $\tilde{Y}_s^{\varepsilon, s, X, Y} = Y$.  According to the definition of $\hat{Y}_t^\varepsilon$,  we have that, for any $k\ge 1$,
$$
\hat{Y}_t^\varepsilon=\tilde{Y}_t^{\varepsilon, k \delta, X_{k \delta}^\varepsilon, Y_{k \delta}^\varepsilon}, \quad t \in[k \delta,(k + 1) \delta].
$$
This implies that for any $ s\in [k \delta,(k + 1) \delta]$,
\begin{align*}
&\EE\left\langle F\left(s/{\varepsilon},X_{k \delta}^{\varepsilon},  \hat{Y}_s^{\varepsilon}\right)-\bar{F}\left(s/{\varepsilon},X_{k \delta}^{\varepsilon}\right), \partial_x U_{\rho}(Z^{\varepsilon}_{k\delta})\right\rangle\\
=& \EE\left\langle F\left(s/{\varepsilon},X_{k \delta}^{\varepsilon},  \tilde{Y}_s^{\varepsilon, k \delta, X_{k \delta}^\varepsilon, Y_{k \delta}^\varepsilon}\right)-\bar{F}\left(s/{\varepsilon},X_{k \delta}^{\varepsilon}\right), \partial_x U_{\rho}(Z^{\varepsilon}_{k\delta})\right\rangle.
\end{align*}
For any fixed $x, y \in H$,  since $\tilde{Y}_s^{\varepsilon, k \delta, x, y}$ is independent of $\mathscr{F}_{k \delta}$, and $X_{k \delta}^\varepsilon, Y_{k \delta}^\varepsilon$ and $Z^{\varepsilon}_{k\delta}$ are $\mathscr{F}_{k \delta}$-measurable, we can obtain
\begin{align*}
& \EE\left\langle F\left(s/{\varepsilon},X_{k \delta}^{\varepsilon},  \tilde{Y}_s^{\varepsilon, k \delta, X_{k \delta}^\varepsilon, Y_{k \delta}^\varepsilon}\right)-\bar{F}\left(s/{\varepsilon},X_{k \delta}^{\varepsilon}\right), \partial_x U_{\rho}(Z^{\varepsilon}_{k\delta})\right\rangle\\
=& \EE \left\{\left\langle \left[\EE F\left(s/{\varepsilon},x,  \tilde{Y}_s^{\varepsilon, k \delta, x, y}\right)-\bar{F}\left(s/{\varepsilon},x\right)\right], \partial_x U_{\rho}(z)\right\rangle 1_{x=X_{k \delta}^{\varepsilon},y=Y_{k \delta}^\varepsilon,z=Z^{\varepsilon}_{k\delta}}\right\}.
\end{align*}
Meanwhile, from the definition of $\tilde{Y}^{\varepsilon, s, x, y}$, we can see that for any $t \in[k \delta/\varepsilon, (k + 1) \delta/\varepsilon]$,
\begin{equation}\label{4.9}
\begin{aligned}
\tilde{Y}_{t\varepsilon}^{\varepsilon, k \delta, x, y} & = e^{[(t\varepsilon-k \delta) / \varepsilon]B}y+\frac{1}{\varepsilon} \int_{k \delta}^{t\varepsilon} e^{[(t\varepsilon-s) / \varepsilon]B} G\left(s / \varepsilon, x, \tilde{Y}_s^{\varepsilon, k \delta, x, y}\right) ds+\frac{1}{\varepsilon^{1/{\alpha}}}\int_{k \delta}^{t\varepsilon} e^{[(t\varepsilon-s) / \varepsilon]B} dZ_s \\
& = e^{\left[t-\left(k \delta/ \varepsilon\right)\right]B}y+\int_{k \delta/\varepsilon}^t e^{(t-s)B}G\left(s, x, \tilde{Y}_{s\varepsilon}^{\varepsilon, k \delta, x, y}\right) \mathrm{d}s+\int_{k \delta/\varepsilon}^t e^{(t-s)B}d\hat{Z}_s,
\end{aligned}
\end{equation}
where $\hat{Z}_s :=\varepsilon^{-1/\alpha}Z_{s\varepsilon}$. Recall that the solution of the frozen equation satisfies
\begin{equation}\label{4.10}
 Y_t^{k\delta/\varepsilon, x, y}=e^{\left[t-\left(k \delta/ \varepsilon\right)\right]B} y+\int_{k \delta/\varepsilon}^t e^{(t-s)B} G(s, x, Y_s^{k \delta/\varepsilon, x, y}) ds+\int_{k \delta/\varepsilon}^t e^{(t-s)B} dZ_s.
\end{equation}
From the uniqueness of the solution of Eq.(\ref{4.9}) and Eq.(\ref{4.10}),  it implies that $\{\tilde{Y}_t^{\varepsilon, k \delta, x, y}\}_{t \in[k \delta, (k + 1) \delta]}$ and $\{Y_{t/\varepsilon}^{k\delta/\varepsilon, x, y}\}_{t \in[k \delta, (k + 1) \delta]}$ are identically distributed.  The Markov property and \eqref{Wergodicity} yeld that
\begin{align*}
&I^{\vare}_{31}(T) \\
\leq & C_{p,T}\!\!\sum_{k = 0}^{\lfloor T / \delta \rfloor - 1} \!\Big| \!\int_{k \delta}^{(k + 1) \delta}\!\! \EE \!\left\{\left\langle \left[\EE F(s/\varepsilon,x,  Y_{s/\varepsilon}^{k \delta/\varepsilon, x, y})\!-\!\bar{F}(s/{\varepsilon},x)\right], \partial_x U_{\rho}(z)\right\rangle 1_{x=X_{k \delta}^{\varepsilon},y=Y_{k \delta}^\varepsilon,z=Z^{\varepsilon}_{k\delta}}\right\} ds\Big|\\
\leq & C_{p,T}\sum_{k = 0}^{\lfloor T / \delta \rfloor - 1} \int_{k \delta}^{(k + 1) \delta} \EE \left[(1+|X_{k \delta}^{\varepsilon}|+|Y_{k \delta}^{\varepsilon}|)|Z^{\varepsilon}_{k\delta}|^{p-1}\right] e^{-(\beta_1-L_G)(s-k\delta)/(2\varepsilon)}ds\\
\leq & C_{p,T}(1+|x|^p+|y|^p)\varepsilon/\delta.
\end{align*}
Therefore, it is asserted that \eqref{I3} holds. The proof is complete.
\end{proof}

\subsection{Proof of Theorem \ref{MR2}}
Since $G(\cdot,x,y)$ is $\tau_2$-periodic as stated in Assumption \ref{A4},  one has that $P^x_{s,t}$ is also $\tau_2$-periodic, i.e.,
$$P^x_{s,t}\varphi(y)=P^x_{s+\tau_2,t+\tau_2}\varphi(y),\quad -\infty<s\leq t<+\infty,\quad \varphi \in C_b(H).$$
The uniqueness of the evolution system of measures for $\{P^x_{s,t}\}_{t\geq s}$, as outlined in Proposition \ref{pp}, one can verify that $\left\{\mu_t^x\right\}_{t \in \mathbb{R}}$ is $\tau_2$-periodic, that is,
$$
\mu_{t + \tau_2}^x = \mu_t^x,\quad \forall t \in \mathbb{R}.
$$
This together with the $\tau_1$-periodicity of $F(\cdot,x,y)$, we finally obtain $\bar{F}(t,x)=\int_{H}F(t,x,y)\mu^x_t(dy)$
is $\tau$-periodic, where $\tau=m_2\tau_1=m_1\tau_2$ for some $m_1,m_2\in \mathbb{N}_{+}$. Then we can construct the averaged coefficient $\bar{F}_P(x)$ defined in \eqref{FP}. Since $\bar{F}(t,x)$ is Lipschitz continuous by \eqref{barc1}, we have
\begin{align}
|\bar F_P(x_{1})-\bar F_P(x_{2})|
\leq\frac{1}{\tau} \int_0^{\tau} |\bar{F}(t, x_1)-\bar{F}(t, x_2)| d t \leq  C|x_1-x_2|.\label{LipFP}
\end{align}
Therefore, it is easy to check that for any $x\in H$, Eq.(\ref{AV3}) has an unique mild solution $\bar{X}_t$, i.e.,
$$
\bar{X}_t = e^{t A} x + \int_0^t e^{(t - r) A} \bar{F}_P\left( \bar{X}_r\right) dr + \int_0^t e^{(t - r) A} d L_r,\quad t\geq 0.
$$

Following the same argument as used in Lemmas \ref{L3.1} and \ref{L3.9}, we have the following estimates.
\begin{lemma}
For any $p\in [1,\alpha)$ and $T>0$, there exists constant $C_T>0$ such that
\begin{align}
&\sup_{t\in [0, T]}\left(\mathbb{E}|\bar{X}_{t}|^{p}\right)^{1/p}\leq C_{T}(1+|x|),\label{FianlEE2}\\
&\int^{T}_0\left[\mathbb{E}|\bar{X}_t-\bar{X}_{t(\delta)}|^p\right]^{1/p}dt\leq C_T\delta^{\frac{\theta }{2}}(1+|x|),\label{IncbarX2}
\end{align}
where $\theta$ is the constant in Assumption \ref{A2}.
\end{lemma}

\vspace{0.2cm}
We are in a position to give a detailed proof of Theorem \ref{MR2}.

\begin{proof}[Proof of Theorem $\ref{MR2}$] Using an elementary result (see e.g. \cite[Lemma 5.1]{SWX2024}): Let $| h|\leq M $ be a  $\tau$-periodic function on $\mathbb{R}$, one has that, for any $ T > 0 $,
$$
\sup_{a \in \mathbb{R}} \left| \frac{1}{T} \int_a^{T + a} h(s) \, ds - \frac{1}{\tau} \int_0^{\tau} h(s) \, ds \right| \leq \frac{2 \tau M}{T}.
$$
This and the definition of $\bar{F}_P$ in \eqref{FP}  yield that
\begin{align}
\frac{1}{T}\left|\int^{t+T}_t [\bar{F}(s,x)-\bar{F}_P(x)]ds\right|\leq \frac{C(1+|x|)}{T}.\label{F5.15}
\end{align}

Note that for any $t\in [0,T]$,
$$
\begin{aligned} \bar{X}_t^{\varepsilon}-\bar{X}_t & =\int_0^t e^{(t-s) A}\left[\bar{F}\left(s/\varepsilon, \bar{X}^{\varepsilon}_s\right)-\bar{F}_P\left(\bar{X}_s\right)\right] d s \\
& =\int_0^t e^{(t-s) A}\left[\bar{F}\left(s / \varepsilon, \bar{X}_s^{\varepsilon}\right)\!-\!\bar{F}\left(s / \varepsilon,\bar X_s\right)\right] d s\!+\!\int_0^t e^{(t-s) A}\left[\bar{F}\left(s / \varepsilon,\bar X_s\right)\!-\!\bar{F}_P\left(\bar{X}_s\right)\right] d s.
\end{aligned}
$$
Then by \eqref{barc1}, we can obtain that for any $p\in [1,\alpha)$,
$$
\begin{aligned}
&\sup_{s \in [0, t]}\mathbb{E}\left|\bar{X}_s^\varepsilon - \bar{X}_s\right|^p \\
\leq &C_{p,T}\int_0^t \mathbb{E}\left|\bar{X}_s^\varepsilon - \bar{X}_s\right|^p ds+ C_{p,T}\sup_{s \in [0, t]}\mathbb{E}\left|\int_0^s e^{(s - r)A}\left[\bar{F}\left(r / \varepsilon, \bar{X}_r\right) - \bar{F}_P\left(\bar{X}_r\right)\right] dr\right|^p.
\end{aligned}
$$
Gronwall's inequality yields that
\begin{align*}
&\sup_{t\in [0,T]}\mathbb{E}\left|\bar{X}_t^\varepsilon - \bar{X}_t\right|^p \\
\leq&  C_{p,T}\sup_{t\in [0,T]}\mathbb{E}\left|\int_0^t e^{(t - s)A}\left[\bar{F}\left(s / \varepsilon, \bar{X}_s\right) - \bar{F}_P\left(\bar{X}_s\right)\right] ds\right|^p\\
\leq&  C_{p,T}\sup_{t\in [0,T]}\EE\left|\int_0^t e^{(t - s(\delta))A}[\bar{F}(s/\varepsilon,\bar{X}_{s(\delta)})-\bar{F}_P(\bar{X}_{s(\delta)})]\,ds\right|^p\\
&+C_{p,T}\EE\left[\int_0^T|\bar{F}(s/\varepsilon,\bar{X}_{s})-\bar{F}(s/\varepsilon,\bar{X}_{s(\delta)})|
+|\bar{F}_P(\bar{X}_{s})-\bar{F}_P(\bar{X}_{s(\delta)})|\,ds\right]^p\\
&+C_{p,T}\sup_{t\in [0,T]}\EE\left[\int_0^t \left|\left(e^{(t - s)A}-e^{(t - s(\delta))A}\right)\left(\bar{F}(s/\varepsilon,\bar{X}_{s})-\bar{F}_P(\bar{X}_{s})\right)\right|\,ds\right]^p\\
=:& J^{\vare}_1(T)+J^{\vare}_2(T)+J^{\vare}_3(T).
\end{align*}
Here, $\delta\in (0,1]$ depends on $\varepsilon$, which will be chosen later.

For $J^{\varepsilon}_1(T)$,  \eref{F5.15} implies that
\begin{align}
J^{\varepsilon}_1(T)\leq& C_T\sup_{t\in [0,T]}\EE\left|\sum^{\lfloor t/\delta\rfloor-1}_{k=0}    \int^{(k+1)\delta}_{k\delta}    e^{(t-k\delta)A}[\bar{F}(s/\varepsilon,\bar{X}_{k\delta})  -  \bar{F}_P(\bar{X}_{k\delta})]  \,ds
\right|^p\nonumber\\
& +  C_T\sup_{t\in [0,T]}\EE\left|\int^{t}_{\lfloor t/\delta\rfloor\delta}  e^{(t-\lfloor t/\delta\rfloor\delta)A}[\bar{F}(s/\varepsilon,\bar{X}_{[t/\delta]\delta})  -  \bar{F}_P(\bar{X}_{\lfloor t/\delta\rfloor\delta})]  \,ds\right|^p\nonumber\\
\leq& C_T\sup_{t\in [0,T]}\EE\left[\sum^{\lfloor t/\delta\rfloor-1}_{k=0}\delta\left|\frac{\varepsilon}{\delta}\int^{\frac{k\delta}{\varepsilon}+\frac{\delta}
{\varepsilon}}_{\frac{k\delta}{\varepsilon}}\left[\bar{F}(s,\bar{X}_{k\delta})-\bar{F}_P(\bar{X}_{k\delta})\right]\,ds\right|\right]^p+C_T(1+|x|^p)\delta^p\nonumber\\
\leq& C_T\left[1+\sup_{t\in [0,T]}\EE|\bar X_t|^p\right]\left(\varepsilon/\delta\right)^p+C_T(1+|x|^p)\delta^p\nonumber\\
\leq& C_T\left(1+|x|^p\right)\left[\left(\varepsilon/\delta\right)^p+\delta^p\right].\label{J1}
\end{align}

For $J^{\varepsilon}_2(T)$, Minkowski's inequality, \eqref{LipFG}, \eqref{LipFP} and \eref{IncbarX2} yield that
\begin{align}
J^{\varepsilon}_2(T)\leq C_T\left[\int_0^T \left[\EE|\bar{X}_{s}-\bar{X}_{s(\delta)}|^p\right]^{1/p}ds\right]^p
\leq C_T (1+|x|^p)\delta^{\theta p/2}.\label{J2}
\end{align}

For $J^{\varepsilon}_3(T)$, by \eqref{rr} and \eqref{FianlEE2}, we have
\begin{align}
J^{\varepsilon}_3(T)\leq& C_T\sup_{t\in [0,T]}\EE\left|\int_0^t \delta^{\theta/2}(t-s)^{-\theta/2}(1+|\bar{X}_{s}|)ds\right|^p\nonumber\\
\leq& C_T\delta^{p\theta/2}\sup_{t\in [0,T]}\left(1+\EE|\bar{X}_{t}|^p\right)\nonumber\\
\leq& C_T\delta^{p\theta/2} (1+|x|^p).\label{J3}
\end{align}

By \eref{J1}-\eref{J3},  one has
\begin{align*}
\sup_{t\in [0,T]}\mathbb{E}\left|\bar{X}_t^\varepsilon - \bar{X}_t\right|^p \leq C_T\left(1+|x|^p\right)\left(\varepsilon/\delta+\delta^{\theta/2}\right)^p.
\end{align*}
Taking $\delta=\vare^{\frac{2}{\theta+2}}$, we get
\begin{align}
\sup_{t\in [0,T]}\mathbb{E}\left|\bar{X}_t^\varepsilon - \bar{X}_t\right|^p \leq C_T\left(1+|x|^p\right)\vare^{\frac{p\theta}{2+\theta}}.\label{F5.19}
\end{align}
Note that $\frac{\theta(p-1)}{\theta(p-1)+2}<\frac{p\theta}{2+\theta}$,  \eqref{F5.19} and \eref{R1} imply the desired result hold.
\end{proof}

\subsection{Proof of Theorem \ref{MR3}}

According to Assumptions \ref{A3} and \ref{A5}, it is easy to check that for any $x_i, y_i \in \mathbb{R}^n$, $i=1,2$,
\begin{align}\label{LiptildeFG}
\begin{split}
&\left|\tilde{F}(x_1,y_1)-\tilde{F}(x_2,y_2)\right| \leq C(|x_1-x_2|+|y_1-y_2|),\\
&\left|\tilde{G}(x_1,y_1)-\tilde{G}(x_2,y_2)\right| \leq C|x_1-x_2|+L_G|y_1-y_2|.
\end{split}
\end{align}
The estimate \eref{LiptildeFG} implies that the equation
\begin{align}
d\tilde Y_{t}=B\tilde Y_{t}dt+\tilde{G}(x,\tilde Y_{t})dt+d Z_t,\quad \tilde Y_{0}=y\label{F5.10}
\end{align}
admits an unique mild solution $\{\tilde Y_{t}^{x,y}\}_{t\geq 0}$. Moreover, using $L_G<\beta_1$, it is easy to prove
$$
\sup_{t\geq 0}\left[\EE|\tilde{Y}^{x,y}_{t}|^p\right]^{1/p}\leq C(1+|x|+|y|),\quad \forall p\in [1,\alpha),
$$
and for any $t \ge 0$, $x_i,y_i\in H$ ($i=1,2$), we have
\begin{align}
\left|\tilde Y_{t}^{x_1,y_1}-\tilde Y_{t}^{x_2,y_2}\right| \le C\exp\left\{-\frac{(\beta_1-L_G) t}{2}\right\}
|y_1-y_2|+C|x_1-x_2|,\label{barErgodicity}
\end{align}

Let $\{\tilde{P}^{x}_{t}\}_{t\geq 0}$ be the transition semigroup of $\tilde Y_{t}^{x,y}$. We have the following statements:

(i) $\{\tilde P^x_t\}_{t\geq 0}$ admits a unique invariant measure $\mu^x$ satisfying
	\begin{align}
		\int_{H}|y|^{p}\mu^x(dy)\leq C_p(1+|x|^p),\quad \forall p\in [1,\alpha).\label{F3.17}
	\end{align}

(ii) For any Lipschitz function $\varphi$ on $H$, we have for any $t>0$,
\begin{align}
\left| \tilde P^{x}_t \varphi(y)-\mu^{x}(\varphi)\right|\leq C\|\varphi\|_{\rm Lip} \exp\left\{-\frac{(\beta_1-L_G) t}{2}\right\}(1+|x|+|y|), \label{ergodicity1}
\end{align}
where $C>0$ and $\|\varphi\|_{\rm Lip}:=\sup_{x\neq y\in H}|\varphi(x)-\varphi(y)|/|x-y|$.

The estimates \eref{barErgodicity} and \eref{ergodicity1} imply that for any $x_1,x_2 \in H$,
\begin{align}
|\mu^{x_1}(\varphi)-\mu^{x_2}(\varphi)|\le C\|\varphi\|_{Lip}|x_1-x_2|.\label{Lipmux}
\end{align}

In this situation, we recall the corresponding averaged coefficient $\bar{F}_A(x)$ defined in \eqref{FA}.
By \eref{LiptildeFG} and \eref{Lipmux}, we have that
\begin{align*}
|\bar F_A(x_{1})-\bar F_A(x_{2})|
=&\left|\int_{H} \tilde F(x_{1},y)\mu^{x_1}(dy)-\int_{H} \tilde F(x_{2},y)\mu^{x_2}(dy)\right|\nonumber\\
\leq&\left|\int_{H} \tilde F(x_{1},y)\mu^{x_1}(dy)-\int_{H} \tilde F(x_{2},y)\mu^{x_1}(dy)\right|\nonumber\\
&+\left|\int_{H} \tilde F(x_{2},y)\mu^{x_1}(dy)-\int_{H} \tilde F(x_{2},y)\mu^{x_2}(dy)\right|\nonumber\\
\le& C|x_1-x_2|.
\end{align*}
Therefore, it is easy to prove the following well-posedness of the solution of Eq.(\ref{AV3}) and a prior estimates of the solution.

\begin{lemma} \label{WFAVE}
For any $x\in H$, Eq.(\ref{AV3}) has an unique mild solution $\bar{X}_t$, i.e.,
$$
X_t = e^{t A} x + \int_0^t e^{(t - r) A} \bar{F}_A( X_r) dr + \int_0^t e^{(t - r) A} d L_r,\quad t\geq 0.
$$
Moreover, for any $p\in [1,\alpha)$ and $T>0$, there exists constant $C_T>0$ such that
\begin{align}
&\sup_{t\in [0, T]}\left(\mathbb{E}|X_{t}|^{p}\right)^{1/p}\leq C_{T}(1+|x|),\label{FianlEE}\\
&\int^{T}_0\left[\mathbb{E}|X_t-X_{t(\delta)}|^p\right]^{1/p}dt\leq C_T\delta^{\frac{\theta }{2}}(1+|x|),\label{IncbarX}
\end{align}
where $\theta$ is the constant in Assumption \ref{A2}.
\end{lemma}

Denote $\{\tilde Y_{t}^{s,x,y}\}_{t\geq s}$ is the solution of the equation
\begin{align}
d\tilde Y_{t}^{s,x,y}=B\tilde Y_{t}^{s,x,y}dt+\tilde{G}(x,\tilde Y_{t}^{s,x,y})dt+d Z_t,\quad \tilde Y_{s}^{s,x,y}=y.\label{SFrE}
\end{align}

\begin{lemma}\label{lem5.3}
For any $T>0$, we have
\begin{align*}
\sup_{t\geq 0}\frac{1}{T}\int^{t+T}_{t}\EE|Y_r^{t,x,y}-\tilde Y_{r}^{t,x,y}|dr
\leq C(1+|x|+|y|)\sup_{t\geq 0}\frac{1}{T}\int^{t+T}_{t}\left[\int^r_0 e^{-(r-s)\beta_1}\phi_2(s)ds\right]dr.
\end{align*}
\end{lemma}

\begin{proof}
For fixed $t\geq 0$, denote $V_r=:Y_r^{t,x,y}-\tilde Y_{r}^{t,x,y}$, $r\geq t$. Then it is easy to obtain
\begin{align*}
\EE|V_r|=&\EE\left|\int^r_t e^{(r-s)B}\left[G(s, x, Y_s^{t, x, y})-\tilde G(s, x, Y_s^{t, x, y})\right]ds\right|\\
\leq &\EE\left|\int^r_t e^{(r-s)B}\left[G(s, x, Y_s^{t,x, y})- \tilde G(s, x, Y_s^{t,x, y})\right]ds\right|\\
&+\EE\left|\int^r_t e^{(r-s)B}\left[\tilde G(s, x, Y_s^{t, x, y})- \tilde G(s, x, \tilde Y_s^{t,x, y})\right]ds\right|\\
\leq &\int^r_t e^{-(r-s)\beta_1}\phi_2(s)\EE(1+|x|+|Y_s^{t,x, y}|)ds+\int^r_t e^{-(r-s)\beta_1}L_G\EE|V_s|ds.
\end{align*}
 Fubini's theorem yields
\begin{align*}
\int^{t+T}_{t}\EE|V_r|dr
\leq &L_G\int^{t+T}_t\!\!\int^r_t \!\!e^{-(r-s)\beta_1}\EE|V_s|dsdr\!+\!\int^{t+T}_{t}\!\!\int^r_t \!\!e^{-(r-s)\beta_1}\phi_2(s)\EE(1\!+\!|x|\!+\!|Y_s^{t,x, y}|)dsdr\\
\leq &L_G\int^{t+T}_{t}\!\!\int^{t+T}_s \!\!e^{-(r-s)\beta_1}\EE|V_s|drds\!+\!C(1\!+\!|x|\!+\!|y|)\int^{t+T}_{t}\!\!\int^r_t \!\!e^{-(r-s)\beta_1}\phi_2(s)dsdr\\
\leq &\frac{L_G}{\beta_1}\int^{t+T}_{t}\EE|V_s|ds\!+\!C(1\!+\!|x|\!+\!|y|)\int^{t+T}_{t}\!\!\int^r_t \!\!e^{-(r-s)\beta_1}\phi_2(s)dsdr.
\end{align*}
By the condition $L_G<\beta_1$, we have
\begin{align*}
\int^{t+T}_{t}\EE|V_r|dr
\leq C(1+|x|+|y|)\int^{t+T}_{t}\!\!\int^r_0\!\! e^{-(r-s)\beta_1}\phi_2(s)dsdr,
\end{align*}
which implies the desired result. The proof is complete.
\end{proof}

\begin{remark}\label{R5.4}
Refer to \cite[Remark 4.2]{SWX2024}, the definition of $\phi_2$ implies $\lim_{r\rightarrow +\infty}\int^r_0 e^{-(r-s)\beta_1}\phi_2(s)ds=0$. Then using the fact that $\lim_{T\to +\infty}\sup_{t\geq0}\frac{1}T\int_{t}^{t + T}\phi(s)ds=0$, for any local integrable function $\phi\geq 0$ satisfying $\phi(T)\to 0$ as $T\to\infty$, thus it follows
$$
\lim_{T\to +\infty}\sup_{t\geq 0}\frac{1}{T}\int^{t+T}_{t}\!\!\int^r_0 \!\!e^{-(r-s)\beta_1}\phi_2(s)dsdr=0.
$$
\end{remark}

\vspace{0.2cm}
We are in a position to give a detailed proof of Theorem \ref{MR3}.

\begin{proof}[Proof of Theorem $\ref{MR3}$]
Using \eref{Wergodicity}, \eqref{F3.17}, \eref{ergodicity1}, Lemma \ref{lem5.3} and Assumption \ref{A5}, we have
\begin{align}
&\frac{1}{T}\Big|\int^{t+T}_t [\bar{F}(s,x)-\bar{F}_A(x)]ds\Big|\nonumber\\
=& \frac{1}{T}\Big|\int^{t+T}_t \left[\int_{H}F(s,x,y)\mu^x_s(dy)-\int_{H}\tilde F(x,y)\mu^x(dy)\right]ds\Big|\nonumber\\
\le &\frac{1}{T}\Big|\int^{t+T}_t \left[\int_{H}F(s,x,y)\mu^x_s(dy)-\int_{H} F(s,x,y)\mu^x(dy)\right]ds\Big|\nonumber\\
&+\frac{1}{T}\Big|\int^{t+T}_t \left[\int_{H}F(s,x,y)\mu^x(dy)-\int_{H}\tilde F(x,y)\mu^x(dy)\right]ds\Big|\nonumber\\
\le &\frac{1}{T}\Big|\int^{t+T}_t \left[\int_{H}F(s,x,y)\mu^x_s(dy)-\EE F(s,x,Y^{t,x,\textbf{0}}_s)\right]ds\Big|\nonumber\\
&+\frac{1}{T}\Big|\int^{t+T}_t \left[\EE F(s,x,Y^{t,x,\textbf{0}}_s)-\EE F(s,x,\tilde Y^{t,x,\textbf{0}}_s)\right]ds\Big|\nonumber\\
&+\frac{1}{T}\Big|\int^{t+T}_t \left[\EE F(s,x,\tilde Y^{t,x,\textbf{0}}_s)-\int_{H} F(s,x,y)\mu^x(dy)\right]ds\Big|\nonumber\\
&+\frac{1}{T}\Big|\int^{t+T}_t \left[\int_{H}F(s,x,y)\mu^x(dy)-\int_{H}\tilde F(x,y)\mu^x(dy)\right]ds\Big|\nonumber\\
\leq& C(1+|x|)\frac{1}{T}\int^{t+T}_t \left[e^{-(\beta_1-L_G) (s-t)/2}+\int^s_0 e^{-(s-u)\beta_1}\phi_2(u)du\right]ds\nonumber\\
&+C\int_{H} (1+|x|+|y|)\mu^x(dy)\phi_1(T)\nonumber\\
\leq&C(1+|x|)(1/T+\phi_1(T)+\tilde\phi_2(T)),\label{F4.19}
\end{align}
where $\tilde\phi_2(T)$ is defined in \eqref{phi2}, which satisfies $\lim_{T\to \infty}\tilde\phi_2(T)=0$ by Remark \ref{R5.4}.

Using \eqref{F4.19} and following the same argument as used in the proof of Theorem \ref{MR2}, we can easily obtain
\begin{align*}
\sup_{t\in [0,T]}\mathbb{E}\left|\bar{X}_t^\varepsilon - X_t\right|^p \leq C_T\left(1+|x|^p\right)\left[\varepsilon/\delta+\phi_1(\delta/\varepsilon)+\tilde\phi_2(\delta/\varepsilon)+\delta^{\theta/2}\right]^p.
\end{align*}
Taking $\delta=\varepsilon^{2/(\theta+2)}$, we get
\begin{align*}
\sup_{t\in [0,T]}\EE|\bar{X}_t^\varepsilon-X_{t}|^{p}\leq C_T\left(1+|x|^p\right)\left[\vare^{\theta/(\theta+2)}
+\phi_1(\vare^{-\theta/(\theta+2)})+\tilde\phi_2(\vare^{-\theta/(\theta+2)})\right]^p.
\end{align*}
Therefore, the desired assertion immediately follows from the above statement and \eref{R1}. The proof is complete.

\end{proof}

\section{Application to examples}
In this section, we will present a concrete example to illustrate our results. Considering the following non-linear stochastic heat equation on $D=[0,\pi]$ with Dirichlet boundary conditions:
\begin{equation}\left\{\begin{array}{l}\label{Ex}
\displaystyle
dX^{\varepsilon}(t,\xi)=\left[\Delta X^{\varepsilon}(t,\xi)
+f(t/\vare,X^{\varepsilon}(t,\xi), Y^{\varepsilon}(t,\xi))\right]dt+dL(t,\xi),\vspace{2mm}\\
\displaystyle dY^{\varepsilon}(t,\xi)=\frac{1}{\varepsilon}\left[\Delta Y^{\varepsilon}(t,\xi)+g(t/\vare,X^{\varepsilon}(t,\xi), Y^{\varepsilon}(t,\xi))\right]dt
+\frac{1}{\varepsilon^{1/\alpha}}dZ(t,\xi),\vspace{2mm}\\
X^{\vare}(t,\xi)=Y^{\vare}(t,\xi)=0, \quad t> 0,\quad \xi\in\partial D,\vspace{2mm}\\
X^{\vare}(0,\xi)=x(\xi),Y^{\vare}(0,\xi)=y(\xi)\quad \xi\in D, x,y\in H,\end{array}\right.
\end{equation}
where $\partial D$ the boundary of $D$. $L$ and $Z$ are two mutually independent cylindrical $\alpha$-stable processes defined on probability space $(\Omega, \mathscr{F},\mathbb{P})$ with $\alpha\in(1,2)$ , that is
$$
L_t=\sum^{\infty}_{k=1}L^{k}_{t}e_k,\quad Z_t=\sum^{\infty}_{k=1}Z^{k}_{t}e_k,\quad t\geq 0,
$$
where $\{L^k\}_{k\ge 1}$ and $\{Z^k\}_{k\ge 1}$ are two sequences of independent one dimensional symmetric $\alpha$-stable processes.

Let $\Delta$ be Laplacian operator  on $H:=L^2(D)$ with Dirichlet boundary assumptions. Taking
$$Ax=Bx=\Delta x,\quad\quad x\in \mathscr{D}(A)=\mathscr{D}(B):=H^{2}(D) \cap H^1_0(D),$$
then $A$ and $B$ are self-adjoint operators and possesses a complete orthonormal system of eigenfunctions, namely
$$e_k(\xi)=(\sqrt{2/\pi})\sin(k\xi),\quad\xi\in D, k\ge 1.$$
Moreover, $Ae_k=Be_k=-k^2 e_k$. Thus Assumption \ref{A1} holds for $\lambda_k=\beta_k=k^2$.

Set $\rho_k=\gamma_k=1$, $k\ge1$.  It is easy to see that
\begin{align*}
\sum^{\infty}_{k=1}\frac{1}{\lambda^{1-\alpha\theta/2}_k}<\infty,\quad \sum^{\infty}_{k=1}\frac{1}{\beta_k}<\infty
\end{align*}
if and only if $0<\theta<1/\alpha<1$.  Assumption \ref{A2} holds for any $\theta\in (0,1/\alpha)$.

The coefficients $F(t,x,y)(\xi):=f(t,x(\xi),y(\xi))$ and $G(t,x,y)(\xi):=g(t,x(\xi),y(\xi))$, $x,y\in L^2(D)$, are defined to be the Nemytskii operators associated with functions $f$ on $[0,+\infty)\times \RR^2$ and $g$ on $\RR\times \RR^2$ which are measurable and  there exist constants $C>0, L_{g}<1$ such that for any $x_i,y_i\in \RR$, $i=1,2$,
\begin{align*}
&\sup_{t\geq 0}|f(t,x_1, y_1)-f(t,x_2, y_2)|\leq C\left(|x_1-x_2| + |y_1-y_2|\right),\\
&\sup_{t\in\RR}|g(t,x_1, y_1)-g(t,x_2, y_2)|\leq C|x_1-x_2| + L_g|y_1-y_2|,\\
&\sup_{t\geq 0}|f(t,x_1,y_1)|\leq C(1+|x_1|+|y_1|),\quad \sup_{t\in\RR}|g(t,x_1,y_1)|\leq C\left(1+|x_1|\right)+L_g|y_1|.
\end{align*}
Then it is easy to check that Assumption \ref{A3} holds with $L_G=L_g<\beta_1$.

Now, we can rewrite stochastic system \eqref{Ex} as follows
\begin{equation*}
 \left\{\begin{array}{l}d X_t^{\varepsilon}=\left[A X_t^{\varepsilon}+F\left(t / \varepsilon, X_t^{\varepsilon}, Y_t^{\varepsilon}\right) d t+d L_t,\quad X_0^{\varepsilon}=x \in H,\right. \\
 d Y_t^\varepsilon=\frac{1}{\varepsilon}\left[BY_t^{\varepsilon}+G\left(t / \varepsilon, X_t^{\varepsilon}, Y_t^{\varepsilon}\right)\right] d t+\frac{1}{\varepsilon^{1 / \alpha}} d Z_t, \quad Y_0^{\varepsilon}=y \in H,\end{array}\right.
\end{equation*}
Theorem \ref{MR1} implies that for any $ x, y \in H $, $p\in (1,\alpha)$ and $ T > 0 $, we have
\begin{align*}
\sup _{t \in[0, T]}\mathbb{E}\left|X_t^{\varepsilon}-\bar{X}_t^{\varepsilon}\right|^p\leq C_{p,T}(1+|x|^p+|y|^p)\varepsilon^{\frac{\theta(p-1)}{\theta(p-1)+2}},
\end{align*}
where $\bar{X}_t^{\varepsilon}$ is the solution of the corresponding averaged equation.

\textbf{Time periodic case:} If we assume there exists $\tau>0$ such that
$$f(t,x_1,y_1):=f(t+\tau,x_1,y_1),\quad t\in \RR_{+},$$
 $$g(t,x_1,y_1):=g(t+\tau,x_1,y_1),\quad t\in \RR,$$
Then $F(\cdot,x,y)$ and $G(\cdot,x,y)$ are $\tau$-periodic obviously and Assumption \ref{A4} holds.
Therefore, for any $ x, y \in H $, $p\in (1,\alpha)$ and $ T > 0 $, Theorem \ref{MR2} yields that
\begin{align*}
\sup _{t \in[0, T]}\mathbb{E}\left|X_t^{\varepsilon}-\bar X_t\right|^p\leq C_T(1+|x|^p+|y|^p)\varepsilon^{\frac{\theta(p-1)}{\theta(p-1)+2}}.
\end{align*}
where $\bar{X}_t$ is the solution of the corresponding averaged equation.

\textbf{Asymptotic convergence case:} If we further assume there exists $\tilde f$ and $\tilde g$ on $\RR^2$ such that
\begin{align*}
\sup_{t\geq 0}\frac{1}{T}\left|\int^{t+T}_t \left[f(s,\xi_1,\xi_2)-\tilde{f}(\xi_1,\xi_2)\right]ds\right|\leq \phi_1(T)(1+|\xi_1|+|\xi_2|)
\end{align*}
and
 $$|g(T,\xi_1,\xi_2)-\tilde g(\xi_1,\xi_2)|\leq \phi_2(T)(1+|\xi_1|+|\xi_2|),$$
where $\phi_i(T)\to 0$ as $T\to \infty$, $i=1,2$.  For any $x,y\in L^2(D)$, define
\begin{align*}
\tilde{F}(x,y)(\xi):=\tilde f(x(\xi),y(\xi)),~~~
\tilde G(x,y)(\xi)=:\tilde g(x(\xi),y(\xi)).
\end{align*}
Thus it follows
\begin{align*}
\sup_{t\geq 0}\frac{1}{T}\left|\int^{t+T}_t \left[F(s,x,y)-\tilde{F}(x,y)\right]ds\right|\leq \phi_1(T)(1+|x|+|y|)
\end{align*}
and
\begin{align*}
|G(T,x,y)-\tilde G(x,y)|\leq \phi_2(T)(1+|x|+|y|).
\end{align*}
Thus Assumption \ref{A5} holds.  For any $ x, y \in H $, $p\in (1,\alpha)$ and $ T > 0 $, Theorem \ref{MR3} implies that
\begin{align*}
\sup _{t \in[0, T]}\mathbb{E}\left|X_t^\varepsilon - X_t\right|^p
\leq  C_{p,T}(1 + |x|^p+ |y|^p)\left[\varepsilon^{\frac{\theta(p-1)}{\theta(p-1)+2}}
+\left(\phi_1(\vare^{-\theta/(\theta+2)})+\tilde\phi_2(\vare^{-\theta/(\theta+2)})\right)^p\right],
\end{align*}
where $X_t$ is the solution of the corresponding averaged equation and $\tilde{\phi}_2$ is defined in \eqref{phi2}.

\vspace{0.3cm}
\textbf{Acknowledgment}. The research of Xiaobin Sun is supported by the NSF of China (Nos.
12271219, 12090010 and 12090011). The research of Yingchao Xie is supported by the NSF of China  (No 12471139) and the Priority Academic Program Development of Jiangsu Higher Education Institutions.

\end{document}